\documentclass
{compositio}

\usepackage{amscd, amsmath, amsthm, amssymb, mathtools, mathrsfs}
\usepackage{color}
\usepackage{here}

\usepackage[normalem]{ulem} 
\usepackage[all]{xy}
\usepackage{extarrows}
\usepackage{graphicx}
\usepackage{bm} 
\usepackage{comment, url}

\usepackage{tikz}
\usetikzlibrary{intersections, calc, arrows.meta}

\usepackage[pagewise, mathlines]{lineno} 
\usepackage{time}

\usepackage{etoolbox}          

\newcommand*\linenomathpatch[1]{%
  \cspreto{#1}{\linenomath}%
  \cspreto{#1*}{\linenomath}%
  \csappto{end#1}{\endlinenomath}%
  \csappto{end#1*}{\endlinenomath}%
}
\newcommand*\linenomathpatchAMS[1]{%
  \cspreto{#1}{\linenomathAMS}%
  \cspreto{#1*}{\linenomathAMS}%
  \csappto{end#1}{\endlinenomath}%
  \csappto{end#1*}{\endlinenomath}%
}

\expandafter\ifx\linenomath\linenomathWithnumbers 
 \let\linenomathAMS\linenomathWithnumbers
 \patchcmd\linenomathAMS{\advance\postdisplaypenalty\linenopenalty}{}{}{}
\else
  \let\linenomathAMS\linenomathNonumbers
\fi

\linenomathpatch{equation}
\linenomathpatchAMS{gather}
\linenomathpatchAMS{multline}
\linenomathpatchAMS{align}
\linenomathpatchAMS{alignat}
\linenomathpatchAMS{flalign}

\makeatletter
\patchcmd{\mmeasure@}{\measuring@true}{
  \measuring@true
  \ifnum-\linenopenaltypar>\interdisplaylinepenalty
    \advance\interdisplaylinepenalty-\linenopenalty
  \fi
  }{}{}
\makeatother

\newtheorem{theoremcounter}{Theorem Counter}[section]

\theoremstyle{definition}
\newtheorem{definition}[theoremcounter]{Definition}
\newtheorem{remark}[theoremcounter]{Remark}

\theoremstyle{plain}
\newtheorem{lemma}[theoremcounter]{Lemma}
\newtheorem{proposition}[theoremcounter]{Proposition}
\newtheorem{corollary}[theoremcounter]{Corollary}

\newtheorem{theorem}[theoremcounter]{Theorem}

\numberwithin{equation}{section}

\newtheorem{thmx}{Theorem} 

\newcommand{\Z}{\mathbb{Z}}
\newcommand{\Q}{\mathbb{Q}}
\newcommand{\R}{\mathbb{R}}
\newcommand{\C}{\mathbb{C}}

\newcommand{\bbH}{\mathbb{H}}

\DeclareMathOperator{\ImNew}{Im}
\renewcommand{\Im}{\ImNew}
\DeclareMathOperator{\ReNew}{Re}
\renewcommand{\Re}{\ReNew}

\DeclareMathOperator{\SL}{SL}
\DeclareMathOperator{\PSL}{PSL}
\DeclareMathOperator{\sgn}{sgn}
\DeclareMathOperator{\Res}{Res}
\DeclareMathOperator{\tr}{tr}
\DeclareMathOperator{\vol}{vol}

\newcommand{\pmat}[1]{\begin{pmatrix}#1\end{pmatrix}}

\newcommand{\smat}[1]{\bigl(\begin{smallmatrix}#1\end{smallmatrix}\bigr)} 

\newcommand{\pmx}[1]{\begin{pmatrix}#1\end{pmatrix}}
\newcommand{\spmx}[1]{{\small \pmx{#1}}}

\newcommand{\ol}{\overline}

\newcommand{\wt}[1]{{\widetilde{#1}}}

\newcommand{\ds}{\displaystyle}

\newcommand{\surj}{\twoheadrightarrow}

\newcommand{\congto}{\overset{\cong}{\to}}

\title[The Rademacher symbols for triangle groups] 
{Modular knots, automorphic forms, and the Rademacher symbols for triangle groups}

\author{Toshiki Matsusaka}
\email{matsusaka@math.kyushu-u.ac.jp} 
\address{Faculty of Mathematics, Kyushu University\\ 
744 Motooka, Nishi-ku, Fukuoka-shi, 819-0395, Fukuoka, Japan
}

\author{Jun Ueki} 
\email{uekijun46@gmail.com}
\address{Department of Mathematics, Faculty of Science, Ochanomizu University\\ 
2-1-1 Otsuka, Bunkyo-ku, 112-8610, Tokyo, Japan} 


\begin{document}

\begin{abstract}
\'{E}.\,Ghys proved that the linking numbers of modular knots and the ``missing'' trefoil $K_{2,3}$ in $S^3$ coincide with the values of a highly ubiquitous function called the Rademacher symbol for $\SL_2\Z$. 
In this article, we replace $\SL_2\Z=\Gamma_{2,3}$ by the triangle group $\Gamma_{p,q}$ for any coprime pair $(p,q)$ of integers with $2\leq p<q$. 
We invoke the theory of harmonic Maass forms for $\Gamma_{p,q}$ to introduce the notion of the Rademacher symbol $\psi_{p,q}$, and provide several characterizations. Among other things, we generalize Ghys's theorem for modular knots around any ``missing'' torus knot $K_{p,q}$ in $S^3$ and in a lens space. 
\end{abstract}

\maketitle

\makeatletter 
\renewcommand{\l@subsubsection}[2]{\@dottedtocline{2}{3.8em}{3.2em}{#1 \hspace{30em}}{}}
\makeatother

\setcounter{tocdepth}{3}
{\small 
\tableofcontents } 


\section{Introduction}


In a celebrated paper ``\emph{Knots and dynamics}" \cite{Ghys2007ICM}, \'{E}tienne Ghys proved a highly interesting result connecting two objects, one coming from low-dimensional topology and the other coming from automorphic forms for $\SL_2\Z$. 
Namely, he proved that \emph{the linking number of a modular knot with the ``missing'' trefoil $K_{2,3}$ coincides with the value of the Rademacher symbol}. 

Now let $(p,q)$ be any coprime pair of integers with $2\leq p<q$ and put $r=pq-p-q$. 
In this article, we invoke the theory of harmonic Maass forms 
to introduce the notion of the Rademacher symbol $\psi_{p,q}$ for the triangle group $\Gamma_{p,q}=\Gamma(p,q,\infty)<\SL_2\R$. 
Then we extend Ghys's theorem to modular knots around the general torus knot $K_{p,q}$ in $S^3$ and its image in the lens space $L(r,p-1)$. 
We also provide several characterizations of the symbol and discuss its several variants. 

In order to make a concrete description, 
let us briefly recollect the case for $\Gamma_{2,3}=\SL_2\Z$. 
The exterior of the trefoil $K_{2,3}$ in $S^3$ is homeomorphic to the set $\{(z,w)\in \C^2 \mid z^3-w^2\neq 0, |z|^2+|w|^2=1\}$, 
the unit tangent bundle $T_1{\rm PSL}_2\Z \backslash \bbH$ of the modular orbifold, 
and the homogeneous space $\SL_2\Z \backslash  \SL_2\R$. 
\emph{The} so-called \emph{geodesic flow} on $S^3-K_{2,3}\cong \SL_2\Z \backslash  \SL_2\R$ is defined by $\varphi^t:M\mapsto M\smat{e^t&0\\0&e^{-t}}$, $t\in \R$ and its closed orbits are called \emph{modular knots}. Each modular knot corresponds to the conjugacy class of primitive hyperbolic elements of $\SL_2\Z$; 
Let $\gamma=\smat{a&b\\c&d}\in \SL_2\Z$ with $a+d>2$ and $c>0$ so that we have $M_\gamma^{-1}\gamma M_\gamma=\smat{\xi_\gamma&0\\0&\xi_\gamma^{-1}}$ with $\xi_\gamma>1$ for some $M_\gamma \in \SL_2\R$. 
Then the corresponding modular knot $C_\gamma$ is explicitly defined by the curve 
\[C_\gamma(t)=M_\gamma \spmx{e^t&0\\0&e^{-t}},\ \ 0\leq t\leq \log \xi_\gamma.\]  
Note in addition that the unique cusp orbit of the unit tangent bundle  $T_1{\rm PSL}_2\Z \backslash \bbH$ is the missing trefoil $K_{2,3}$ in $S^3$. 

The discriminant function $\Delta_{2,3}(z)=q\prod_{n=1}^{\infty} (1-q^n)^{24}$ with $q=e^{2\pi iz}$ and $z\in \bbH$ is a well known modular form of weight 12. 
The Rademacher symbol $\psi_{2,3}: \SL_2\Z\to \Z$ is defined as the function satisfying the transformation law 
\[\log \Delta_{2,3}(\gamma z) - \log \Delta_{2,3}(z) = 12 \log (cz+d) + 2\pi i \psi_{2,3}(\gamma)\]
for every $\gamma=\smat{a&b\\c&d}\in  \SL_2\Z$ and $z\in \bbH$. 
Here, we take branches of the logarithms so that we have ${\rm Im}\log (cz+d)\in [-\pi, \pi)$ and $\log \Delta_{2,3}(z) = 2\pi iz - 24 \sum_{n=1}^\infty \sum_{d|n} d^{-1} q^n$.  

In \cite[Sections 3.2--3.3]{Ghys2007ICM}, Ghys precisely asserts the following; 
\emph{Let $\gamma = \smat{a&b\\c&d} \in \SL_2 \Z$ be a primitive element with $\tr \gamma  > 2$ and $c > 0$. 
Then the linking number is given by}
\[\mathrm{lk} (C_\gamma, K_{2,3}) = \psi_{2,3}(\gamma).\] 
%

The original Rademacher symbol $\Psi_{2,3}:\SL_2\Z\to \Z$ is a class-invariant function initially introduced by Rademacher in his study of Dedekind sums (\cite{Rademacher1956}, see also \cite{RademacherGrosswald1972}). 
This function is highly ubiquitous, as Atiyah \cite[Theorem 5.60]{Atiyah1987MA} proved the (partial) equivalences of seven very distinct definitions and Ghys went further. 
Ghys indeed worked in his 1st proof with the slightly modified function $\psi_{2,3}:\SL_2\Z\to \Z$, 
which is not a class-invariant function but coincides with the original symbol $\Psi_{2,3}$ at every $\gamma \in \SL_2\Z$ with $\tr \gamma>0$. 
In this paper, we proactively take advantage of Ghys's 1st treatment and afterwards discuss the original symbol. 

Ghys's theorem for $(2,3)$ generalizes to any $(p,q)$ in two directions; 
Let $\wt{\SL_2}\R$ denote the universal covering group of $\SL_2\R$ and $\wt{\Gamma}_{p,q}< \wt{\SL_2}\R$ the inverse image of $\Gamma_{p,q}$. 
Let $G_r$ denote the kernel of a surjective homomorphism $\wt{\Gamma}_{p,q} \surj \Z/r\Z$, which is unique up to multiplication by units in $\Z/r\Z$. Then we have the following. 
\begin{itemize} 
\item \cite{Tsanov2013EM} The spaces $\Gamma_{p,q}\backslash \SL_2\R\cong \wt{\Gamma}_{p,q} \backslash \wt{\SL_2}\R$ are homeomorphic to the exterior of a knot $\ol{K}_{p,q}$ in the lens space $L(r,p-1)$, where $\ol{K}_{p,q}$ is the image of a $(p,q)$-torus knot $K_{p,q}$ via the $\Z/r\Z$-cover $h:S^3\surj L(r,p-1)$. 
\item \cite{RaymondVasquez1981}
The space $G_r\backslash \wt{\SL_2}\R$ is homeomorphic to the exterior of the torus knot $K_{p,q}$ in $S^3$. 
\end{itemize} 
Modular knots in $L(r,p-1)-\ol{K}_{p,q}\cong \Gamma_{p,q}\backslash \SL_2\R$ are defined in a similar manner to the case for $\SL_2\Z$, so that each of them corresponds to a conjugacy class of primitive hyperbolic elements of $\Gamma_{p,q}$. 
In order to define the Rademacher symbol $\psi_{p,q}$ for the triangle group $\Gamma_{p,q}$, 
we invoke the theory of \emph{harmonic Maass forms for} Fuchsian groups; 
we construct a harmonic Maass form $E_2^{(p,q),\ast}(z)$ and a mock modular form $E_2^{(p,q)}(z)$ of weight 2 
to define a suitable holomorphic cusp form $\Delta_{p,q}(z)$ of weight $2pq$, that is, $\Delta_{p,q}(z)$ has no poles and zeros on $\bbH$ and has a unique zero of order $r$ at the cusp $i\infty$. 
The Rademacher symbol $\psi_{p,q}:\Gamma_{p,q}\to \Z$ is then defined as a unique function satisfying the transformation law 
\[\log \Delta_{p,q}(\gamma z) - \log \Delta_{p,q}(z) = 2pq \log (cz+d) + 2\pi i \psi_{p,q}(\gamma)\]
for every $\gamma=\smat{a&b\\c&d}\in \Gamma_{p,q}$ and $z\in \bbH$ 
under suitable choices of branches of the logarithms.  
Among other things, we prove the following assertion in Section 4. 
\begin{thmx} 
\label{our-main-theorem1} 
{\rm (I)} Let the notation be as above. 
	Let $\gamma = \smat{a & b \\ c & d} \in \Gamma_{p,q}$ be a primitive element with $\tr \gamma  > 2$ and  
	$c > 0$. Then the linking number of the modular knot $C_\gamma$ and the image of the missing torus knot $\ol{K}_{p,q}$ in $L(r,p-1)$ is given by the Rademacher symbol $\psi_{p,q}(\gamma)$ as
\[ \mathrm{lk} (C_\gamma, \ol{K}_{p,q}) = \dfrac{\,1\,}{r}\psi_{p,q} (\gamma) \ \ \in \dfrac{\,1\,}{r}\Z.\] 

{\rm (II)} In addition, let $C_\gamma'$ be a connected component of the preimage $h^{-1}(C_\gamma)$ via the $\Z/r\Z$-cover $h:S^3-K_{p,q}\surj L(r,p-1)-\ol{K}_{p,q}$. Then the linking number is given by
\[ {\rm lk}(C_\gamma',K_{p,q})=\dfrac{1}{{\rm gcd}(r,\psi_{p,q}(\gamma))}\psi_{p,q}(\gamma) \in \Z.\] 
\end{thmx} 

In the course of the proof, we obtain the following as well. 
\begin{thmx}
\label{Rad-omnibus-theorem1}
The Rademacher symbol $\psi_{p,q}: \Gamma_{p,q} \to \Z$ has the following five $($partial$)$ characterizations. 
\begin{itemize}
\item[$(1)$] Definition $($Subsection \ref{s2-3}$)$: $\log \Delta_{p,q}(\gamma z) - \log \Delta_{p,q}(z) = 2pq \log (cz+d) + 2\pi i\psi_{p,q}(\gamma)$.
\item[$(2)$] Cycle integral $($Subsection \ref{s3-1}$)$: Let $\gamma \in \Gamma_{p,q}$ be as in Theorem \ref{our-main-theorem1}. For the harmonic Maass form $E_2^{(p,q),*}(z)$, we have
$\ds \int_{\ol{S}_\gamma} E_2^{(p,q),*}(z) dz = \frac{\,1\,}{r} \psi_{p,q}(\gamma).$
\item[$(3)$] $2$-cocycle $($Subsection \ref{s3-2}$)$: 
Define a bounded 2-cocycle $W : \SL_2 \R \times \SL_2 \R \to \{-1, 0, 1\}$ by 
$W(\gamma_1, \gamma_2) = \dfrac{1}{2\pi i} \big( \log j(\gamma_1, \gamma_2 z) + \log j(\gamma_2, z) - \log j(\gamma_1 \gamma_2, z) \big)$ 
as in Definition \ref{def.W}. 
Then, $\psi_{p,q}$ is a unique function satisfying $2pqW=-\delta^1\psi_{p,q}$. 
\item[$(4)$] Additive character $($Subection \ref{s3-3}$)$: Let $\chi_{p,q}: \widetilde{\Gamma}_{p,q} \to \Z$ denote the additive character defined by $\chi_{p,q}(\widetilde{S}_p) = -q$ and $\chi_{p,q}(\widetilde{U}_q) = -p$. 
Then $\psi_{p,q}(\gamma) = \chi_{p,q}(\widetilde{\gamma})$ holds for any $\gamma \in \Gamma_{p,q}$ and its standard lift $\tilde{\gamma}$.  
\item[$(5)$] Linking number $($Subsection \ref{s4-2}$)$: Theorem \ref{our-main-theorem1} $($I$)$. 
\end{itemize}
\end{thmx}
The first two properties are described by means of modular forms, while the third and forth are related to the universal covering group. The final property is of low dimensional topology. 

Although $\psi_{p,q}$ is not a class-invariant function, it seems to be rather a natural object in some aspects and easier to treat, as (1), (3), (4) explain.
Besides, we may modify $\psi_{p,q}$ to define several class-invariant functions, namely, 
the original Rademacher symbol $\Psi_{p,q}$, the homogeneous Rademacher symbol $\Psi_{p,q}^{\rm h}$, 
and the modified Rademacher symbol $\Psi_{p,q}^{\rm e}$ with distinct advantages.

Our results mainly concern Ghys's first proof and briefly a half of his second proof. 
Those are comparable with the results of Dehornoy--Pinsky \cite{Dehornoy2015AGT, DehornoyPinsky2018ETDS} on templates and codings related to Ghys's third proof (cf. Subsection \ref{ss.templates}).

We remark that (mock) modular forms for triangle groups are quite less studied than those for congruence subgroups of $\SL_2\Z$, although they would also be of arithmetic interest; 
For instance, Wolfart \cite{Wolfart1983MA} showed that Fourier coefficients of holomorphic modular forms for the triangle group are mostly transcendental numbers (see also \cite{DoranGannonMovasatiShokri2013}). 
Our study would hopefully give 
a new cliff in this direction. 

The rest of the article is organized as follows. 
In Section \ref{s2}, we recollect harmonic Maass form for the triangle groups to construct the harmonic Maass form $E_2^{(p,q),*}(z)$ and the holomorphic cusp form $\Delta_{p,q}(z)$. 
In Section \ref{s3}, we define the Rademacher symbol $\psi_{p,q}$ for $\Gamma_{p,q}$ and prove the equivalence of (1)--(4) in Theorem \ref{Rad-omnibus-theorem1}. In addition, we discuss several variants of $\psi_{p,q}$. 
In Section \ref{s4}, based on Tsanov's group theoretic study, we establish Theorem \ref{our-main-theorem1} on the linking numbers of modular knots in $L(r,p-1)$ and $S^3$, completing a generalization of Ghys's first proof. 
Further more, 
we define knots corresponding to elliptic and parabolic elements to extend the theorem and give a characterization of the modified symbol $\Psi_{p,q}^{\rm e}$ via an Euler cocycle, to justify Ghys's outlined second proof.  
Finally in Section \ref{s5}, we give remarks on templates and codings and on the Sarnak--Mozzochi distribution theorem for $\Gamma_{p,q}$, and attach further problems. 


\section{Harmonic Maass forms for triangle groups}
\label{s2}


In this section, we introduce harmonic Maass forms for a triangle group $\Gamma_{p,q}$. In particular, we construct two important functions $E_2^{(p,q),*}(z)$ and $\Delta_{p,q}(z)$. The function $E_2^{(p,q),*}(z)$ is a unique harmonic Maass form of weight $2$ on $\Gamma_{p,q}$ with polynomial growth at cusps, and the function $\Delta_{p,q}(z)$ is a unique holomorphic cusp form of weight $2pq$ with no poles and zeros on $\bbH$.


\subsection{Triangle groups} 

Let $(p,q)$ be a coprime pair of integers with $2\leq p<q$ and put $r=pq-p-q$ as before. 
In this subsection, we define the triangle group $\Gamma_{p,q}$ as a subset of $\SL_2\R$ and recall several properties. 

Recall that ${\rm PSL}_2\R=\SL_2\R/\{\pm I\}$ acts on the upper half-plain $\bbH=\{z\in \C \mid {\rm Im}(z)>0\}$  via the M\"obius transformation $\gamma z=\frac{az+b}{cz+d}$ for $\gamma=\smat{a&b\\c&d}$ and $z\in \bbH$, so that 
${\rm PSL}_2\R={\rm Isom}^+\bbH$ holds. 
Triangle groups are often defined as a subgroup of ${\rm PSL}_2\R$ 
generated by reflections on the sides of a triangle in $\bbH$. 
However, we here define them as subgroups of $\SL_2\R$ to make our argument simple. Put
\[
T_{p,q} = \pmat{1 & 2 \left(\cos \frac{\pi}{p} + \cos \frac{\pi}{q} \right) \\ 0 & 1},\ 
S_p = \pmat{0 & -1 \\ 1 & 2 \cos \frac{\pi}{p}},\ U_q = \pmat{2 \cos \frac{\pi}{q} & -1 \\ 1 & 0},
\]
so that we have $S_p\!^p=U_q\!^q=-I$ and $T_{p,q}=-U_qS_p$. 
\begin{definition} \label{Gpq-generator}
\emph{The $(p,q)$-triangle group} $\Gamma_{p,q}=\Gamma(p,q,\infty)$ is a subgroup of $\SL_2\R$ generated by 
elements $S_p$ and $U_q$. 
\end{definition} 

This group $\Gamma_{p,q}$ is a Fuchsian group of the first kind. We especially have $\Gamma_{2,3}=\SL_2\Z$.
There is an isomorphism to the amalgamated product 
\[ \Gamma_{p,q} \cong \langle S_p \rangle *_{\langle -I \rangle} \langle U_q \rangle \cong \Z/2p\Z *_{\Z/2\Z} \Z/2q\Z,\] 
which is obtained by applying \cite[Theorem 6]{Serre2003} to a geodesic segment $T = \{e^{\pi i\theta} \mid 1/q \leq \theta \leq 1 - 1/p\} \subset \bbH$. 

We can visualize the group $\Gamma_{p,q}$ by its fundamental domain in $\bbH$. Let $\Delta=\Delta(p,q, \infty)$ denote the triangle with interior angles $\pi/p, \pi/q, 0$ defined by
\[ \Delta(p,q, \infty) = \{z \in \bbH \mid - \cos \frac{\pi}{p} \leq \Re(z) \leq \cos \frac{\pi}{q}, |z| \geq 1 \}. \]
In addition, let $\Delta'=\Delta'(p,q,\infty)$ denote the reflection of $\Delta(p,q,\infty)$ with respect to the geodesic $\{e^{i\theta} \mid 0 < \theta < \pi\}$, that is, we put 
\[ \Delta'(p,q,\infty) = \{ z \in \bbH \mid \spmx{0 & 1 \\ 1 & 0} \overline{z} = \frac{\,1\,}{\overline{z}} \in \Delta(p,q,\infty) \}. \]
Then, the set $D_{p,q} = \Delta(p,q,\infty) \cup \Delta'(p,q,\infty)$ is a fundamental domain of $\Gamma_{p,q}$.

The vertices $a = e^{\pi i \left(1-1/p \right)}, b = e^{\pi i/q}$, and $i\infty$ of $\Delta(p,q,\infty)$ are fixed points of $S_p, U_q$, and $T_{p,q}$, respectively. The first two vertices $a$ and $b$ are called \emph{elliptic points} of $\Gamma_{p,q}$, and $i\infty$ is called a \emph{cusp} of $\Gamma_{p,q}$. 
The stabilizer subgroups of these vertices are given by $(\Gamma_{p,q})_a = \langle S_p \rangle$, $(\Gamma_{p,q})_b = \langle U_q \rangle$, and $(\Gamma_{p,q})_\infty = \pm \langle T_{p,q} \rangle$, respectively. The two sides of the quadrangle $D_{p,q}$ joined at each elliptic point are $\Gamma_{p,q}$-equivalent. Hence, the Riemann surface $\Gamma_{p,q} \backslash \bbH$ has one cusp, two elliptic points, and genus $0$. 
By the Gauss--Bonnet theorem, or by a direct calculation, one may verify that
\[ \vol(\Gamma_{p,q} \backslash \bbH) = \int_{D_{p,q}} \frac{dx dy}{y^2} = \frac{2\pi r}{pq}.\]

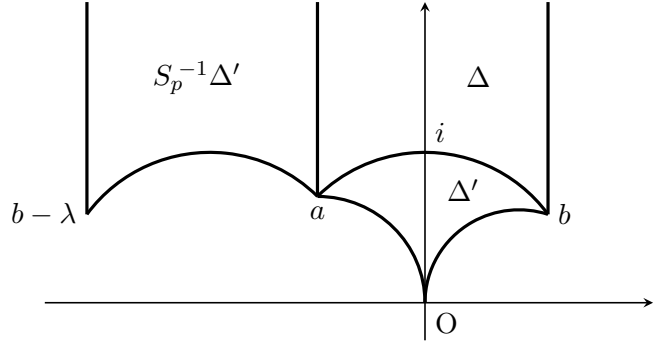
\begin{figure}[H]
\begin{center}
\begin{tikzpicture}
	\draw[->,>=stealth,semithick] (-5,0)--(3,0); 
	\draw[->,>=stealth,semithick] (0,-0.5)--(0,4); 
	\draw (0,0)node[below right]{O}; 
	
	\draw (1.618, 1.176) node[right]{$b$};
	\draw [very thick] (1.618, 1.176) arc (36:135:2);
	\draw (-1.414, 1.414) node[below]{$a$};
	\draw (0,2) node[above right]{$i$};
	
	\draw [very thick] (1.618, 1.176)--(1.618, 4);
	\draw [very thick] (-1.414, 1.414)--(-1.414, 4);
	
	\draw [very thick] (0,0) arc (180:72:1.236);
	\draw [very thick] (0,0) arc (0:90:1.414);
	
	\draw (0.7,3) node {$\Delta$};
	\draw (0.5,1.5) node {$\Delta'$};
	
	\draw [very thick] (-4.446, 1.176)--(-4.446, 4);
	\draw [very thick] (-4.446, 1.176) arc (144:45:2);
	\draw (-4.446, 1.176) node[left]{$b-\lambda$};
	\draw (-3,3) node {$S_p\!^{-1} \Delta'$};
\end{tikzpicture}
\end{center}
\caption{Fundamental domain}
\label{fundamental-dom}
\end{figure}

In general, an element $\gamma \in \Gamma_{p,q}$ is said to be elliptic if $|\tr \gamma|<2$, parabolic if $|\tr \gamma|=2$, and hyperbolic if $|\tr \gamma|>2$. 
In each case, the conjugacy class of $\gamma$ correspond to elliptic points, the cusp, and closed geodesics on $\Gamma_{p,q} \backslash \bbH$ respectively (see also Subsection 3.1). 
An element $\gamma\in \Gamma_{p,q}$ is said to be \emph{primitive} if $\gamma=\pm \sigma^n$ for $\sigma\in \Gamma_{p,q}$ and $n\in \Z$ implies that $n=\pm1$. 

We make use of the following lemma in later calculations. 
\begin{lemma}\label{Chebyshev}
For each integer $n \in \Z$, let $C_n(x) \in \Z[x]$ denote the Chebyshev polynomial of the second kind characterized by $C_{n}(\cos t) = \sin nt/\sin t$, $t\in \R$. Then $C_0(x) = 0$, $C_1(x) = 1$, and $C_{n+1}(x) = 2x C_n(x) - C_{n-1}(x)$ hold.   
The generators $S_p$ and $U_q$ satisfy
\[ S_p\!^{\,n} = \pmat{-C_{n-1} (\cos \frac{\pi}{p}) & -C_n (\cos \frac{\pi}{p}) \\ C_n (\cos \frac{\pi}{p}) & C_{n+1}(\cos \frac{\pi}{p})},\ U_q\!^n = \pmat{C_{n+1} (\cos \frac{\pi}{q}) & -C_n (\cos \frac{\pi}{q}) \\ C_n(\cos \frac{\pi}{q}) & -C_{n-1} (\cos \frac{\pi}{q})}. \]
\end{lemma}

\subsection{Modular forms for $\Gamma_{p,q}$}
In this subsection, we recollect the notions of meromorphic modular forms and harmonic Maass forms for triangle groups together with some properties.  

\subsubsection{Meromorphic modular forms} 
For $\gamma = \smat{a&b\\c&d} \in \SL_2 \R$ and a variable $z\in \bbH$, \emph{the automorphic factor} is defined by $j(\gamma, z) = cz+d$, so that the cocycle condition $j(\gamma_1 \gamma_2, z) = j(\gamma_1, \gamma_2 z) j(\gamma_2, z)$ for any $\gamma_1, \gamma_2 \in \SL_2 \R$ holds. 
For the pair of a function $f : \bbH \to \C$ and an element $\gamma \in \SL_2 \R$, \emph{the slash operator of weight $k \in \Z$} is defined by $(f|_k \gamma) (z) = j(\gamma, z)^{-k} f(\gamma z)$.

\begin{definition}
A meromorphic function $f : \bbH \to \C\cup\{\infty\}$ is called a \emph{meromorphic modular form of weight $k \in \Z$ for $\Gamma_{p,q}$} if the following conditions hold.
	\begin{enumerate}
		\item $f|_k \gamma = f$ for every $\gamma \in \Gamma_{p,q}$.
		\item Put $\lambda = 2 \big(\cos \frac{\pi}{p} + \cos \frac{\pi}{q}\big)$ and $q_\lambda = e^{2\pi i z/\lambda}$. Then $f(z)$ has a Fourier expansion of the form 
		\[
			f(z) = \sum_{n=n'}^\infty a_n q_\lambda^{\,n}, \ a_n\in \C 
		\]
		for some $n'\in \Z$. 
	\end{enumerate} 
If in addition $f$ is holomorphic on $\bbH$ and $a_n=0$ holds for all $n<0$ (resp. $n\leq 0$), then $f$ is called a \emph{holomorphic modular form} (resp. \emph{cusp form}) of weight $k$ for $\Gamma_{p,q}$. 
The space of all holomorphic modular forms of weight $k$ for $\Gamma_{p,q}$ is denoted by $\mathcal{M}_k(\Gamma_{p,q})$. 
\end{definition}

Cauchy's residue theorem yields the following valence formula in a similar manner to \cite[Proposition 3.8]{Koblitz1993GTM}: 

\begin{proposition}[(The valence formula)] \label{valence} 
Let $f$ be a non-zero meromorphic modular form of weight $k$ for $\Gamma_{p,q}$. 
Let $v_P(f)$ denote the order of zero of $f$ at each $z=P$ on $\Gamma_{p,q}\backslash \bbH$ and put $v_\infty(f) = \min\{n \in \Z \mid a_n \neq 0\}$. 
Then \[
		v_\infty (f) + \frac{\,1\,}{p} v_a (f) + \frac{\,1\,}{q} v_b (f) + \sum_{\substack{P \in \Gamma_{p,q} \backslash \bbH \\ P \neq a, b}} v_P(f) = \frac{r}{2pq} k
	\]
holds, where $a = e^{\pi i (1-1/p)}$ and $b = e^{\pi i/q}$ are the fixed points of $S_p$ and $U_q$ respectively. 
\end{proposition}

We use the following lemma later. 
\begin{lemma} \label{lem.M2=0} 
$\mathcal{M}_2(\Gamma_{p,q})=0$. 
\end{lemma} 
\begin{proof} Suppose $0\neq f \in \mathcal{M}_2(\Gamma_{p,q})$ and put $N=v_\infty (f)+\sum_P v_P(f)$, $A=v_a(f)$, $B=v_b(f)$. Then we have $N,A,B\in \Z_{\geq 0}$. In addition, Proposition \ref{valence} yields $N+\frac{A}{p}+\frac{B}{q}=\frac{r}{pq}=\frac{pq-p-q}{pq}$. 
Since $0<\frac{pq-p-q}{pq}<1$, we have $N=0$, hence $p(B+1)+q(A+1)-pq=0$. 
Since $p$ and $q$ are coprime, we may write $A+1=pl$ for some $l \in \Z_{>0}$. 
Now we have $0=p(B+1)+pq(l-1)>0$, hence contradiction. Therefore we have $f=0$. 
\end{proof}

\subsubsection{Harmonic Maass forms} 
The notion of harmonic Maass forms is a generalization of holomorphic modular forms. 
It was introduced by Bruinier--Funke \cite{BrunierFunke2004} to study geometric theta lifts 
and played a crucial role in the study of Ramanujan's mock theta functions. 
It is defined by using $\xi$-differential operators and the hyperbolic Laplace operators. 

\begin{definition} Let $k\in \Z$. For a real analytic function $f:\bbH\to \C$, \emph{the $\xi$-differential operator} $\xi_k$ of weight $k$ is defined by
	\[
		\xi_k f= 2iy^k \overline{\frac{\partial}{\partial \overline{z}} f},
	\]
	where $\partial/\partial \overline{z}$ is Wirtinger's derivative defined by
	\[
		\frac{\partial}{\partial \overline{z}} = \frac{\,1\,}{2} \big(\frac{\partial}{\partial x} + i \frac{\partial}{\partial y} \big).
	\]
	\emph{The hyperbolic Laplace operator} $\Delta_k$ of weight $k$ is defined by
	\[
		\Delta_k = -\xi_{2-k} \circ \xi_k = -y^2 \big(\frac{\partial^2}{\partial x^2} + \frac{\partial^2}{\partial y^2} \big) + iky \big(\frac{\partial}{\partial x} + i \frac{\partial}{\partial y} \big).
	\]
\end{definition} 
A direct calculation yields that  
\[ \xi_k (f|_k \gamma) = (\xi_k f)|_{2-k} \gamma\] holds for any $\gamma \in \Gamma_{p,q}$. 
Hence if $f$ satisfies the modular transformation law $f|_k \gamma = f$ of weight $k$ for every $\gamma \in \Gamma_{p,q}$, then so does $\xi_k f$ of weight $2-k$. 
We also note that if $f$ is a holomorphic function, then $\xi_k f=0$ holds.

\begin{definition} \label{def-harmonic-Maass}
	A real analytic function $f : \bbH \to \C$ is called a \emph{harmonic Maass form} of weight $k \in \Z$ for $\Gamma_{p,q}$ if the following conditions hold. 
	\begin{enumerate}
		\item $f|_k \gamma = f$ for every $\gamma \in \Gamma_{p,q}$.
		\item $\Delta_k f(z) = 0$. 
		\item There exists $\alpha > 0$ such that $f(x+iy) = O(y^\alpha)$ as $y \to \infty$ uniformly in $x \in \R$.
	\end{enumerate}
	The space of all harmonic Maass forms of weight $k$ for $\Gamma_{p,q}$ is denoted by $\mathcal{H}_k(\Gamma_{p,q})$. 
\end{definition}
We remark that in a basic textbook of harmonic Maass forms \cite[Definition 4.2]{BringmannFolsomOnoRolen2017}, for instance, the condition (iii) is replaced by a slightly different condition, namely, (iii') There exists a polynomial $P_f(z) \in \C[q_\lambda^{-1}]$ such that $f(z) - P_f(z) = O(e^{-\varepsilon y})$ as $y \to \infty$ for some $\varepsilon > 0$. 
Whichever condition is chosen, 
we have $\mathcal{M}_k(\Gamma_{p,q}) \subset \mathcal{H}_k(\Gamma_{p,q})$ and the $\xi$-differential operator of weight $k$ induces a linear map $\xi_k: \mathcal{H}_k(\Gamma_{p,q}) \to \mathcal{M}_{2-k}(\Gamma_{p,q})$.
A virtue of our choice (iii) is that the function $E_2^{(p,q),*}(z)$ in Subsection 2.3 will be a harmonic Maass form. 

Let $f \in \mathcal{H}_k(\Gamma_{p,q})$ and suppose $k\neq1$. Then 
a standard argument yields a Fourier expansion 
\[ f(x+iy) = \sum_{n \geq 0} c^+(n) q_\lambda^n + c^-(0) y^{1-k} + \sum_{n<0} c^-(n) y^{-k/2} W_{-\frac{k}{2}, \frac{k-1}{2}} (4\pi |n| \dfrac{y}{\lambda}) e^{2\pi inx/\lambda}, \]
where $c^+(n)$ with $n\geq 0$ and $c^-(n)$ with $n\leq 0$ are complex constants and 
$W_{\mu, \nu}(y)$ denotes \emph{the} so-called \emph{$W$-Whittaker function} (cf.~\cite[Chapter VII]{MagnusOberhettingerSoni1966}). 
If instead $k=1$, then $y^{1-k}$ is replaced by $\log y$.

The holomorphic part $f^+(z) = \sum_{n \geq 0} c^+(n) q_\lambda^n$ 
of each $f\in \mathcal{H}_k(\Gamma_{p,q})$ is called a \emph{mock modular form of weight $k$ for} $\Gamma_{p,q}$. 

On the other hand, we remark that the Fourier coefficients $c^-(n)$ of the remaining non-holomorphic part are closely related to a function  $\xi_kf \in \mathcal{M}_{2-k}(\Gamma_{p,q})$ called \emph{the shadow of} the mock modular form $f^+$. In fact, we have 
\[\xi_k f(z) = (1-k) \overline{c^-(0)} - \sum_{n > 0} \overline{c^-(-n)} \big(\dfrac{4\pi n}{\lambda}\big)^{\frac{2-k}{2}} q_\lambda^n.\]


\subsection{The harmonic Maass form $E_2^{(p,q),\ast}(z)$ of weight $2$}

In this subsection, we construct a harmonic Maass form $E_2^{(p,q),\ast}(z)$ and 
a mock modular form $E_2^{(p,q)}(z)$ of weight 2 for $\Gamma_{p,q}$ with explicit descriptions. 
For this purpose, we first recollect the notion of the Eisenstein series $E_{2k}^{(p,q)}(z,s)$ of even weight for $\Gamma_{p,q}$, that yields most basic examples of harmonic Maass forms. 
We refer to Iwaniec's book \cite{Iwaniec2002GSM} and Goldstein's paper \cite{Goldstein1973Nagoya} for some properties, but we rather follow a standard recipe of mock modular forms. 

\subsubsection{The Eisenstein series} 
Recall that the triangle group $\Gamma_{p,q} < \SL_2 \R$ is a Fuchsian group with finite covolume $\mathrm{vol}(\Gamma_{p,q} \backslash \bbH) = 2\pi r/pq$ and the stabilizer subgroup of the unique cusp $i\infty$ is given by $(\Gamma_{p,q})_{i\infty} = \pm \langle T_{p,q} \rangle$. 

Let $\lambda = 2 (\cos \frac{\pi}{p} + \cos \frac{\pi}{q})$ as before and put $\sigma = \smat{\lambda^{1/2} & 0 \\ 0 & \lambda^{-1/2}} \in \SL_2\R$, so that 
$\sigma$ is a scaling matrix of the cusp $i\infty$, that is, $\sigma i\infty = i\infty$ and $\sigma^{-1} T_{p,q} \sigma = \smat{1 & 1 \\ 0 & 1}$ hold. 

\begin{definition}
Let $k$ be an integer. For $z \in \bbH$ and $s \in \C$ with $\Re(s) > 1$, the \emph{real analytic Eisenstein series} of weight $2k$ for $\Gamma_{p,q}$ is defined by
\begin{align*}
	E_{2k}^{(p,q)} (z,s) &= \sum_{\gamma \in (\Gamma_{p,q})_\infty \backslash \Gamma_{p,q}} {\rm Im}(z)^{s-k} \bigg|_{2k} (\sigma^{-1} \gamma)\\
		&= \frac{\,1\,}{\lambda^s} \sum_{\gamma \in (\Gamma_{p,q})_\infty \backslash \Gamma_{p,q}} \frac{{\rm Im}(\gamma z)^{s-k}}{j(\gamma, z)^{2k}}.
\end{align*}
\end{definition}

For each $s$ with ${\rm Re}(s) > 1$, as a function in $z$, this series converges absolutely and uniformly on compact subsets of $\bbH$. 
By the definition, $(E_{2k}^{(p,q)}|_{2k} \gamma) (z,s) = j(\gamma, z)^{-2k} E_{2k}^{(p,q)}(\gamma z,s) = E_{2k}^{(p,q)}(z,s)$ holds for any $\gamma \in \Gamma_{p,q}$. 
By the commutativity $\xi_k (f|_k \gamma) = (\xi_k f) |_{2-k} \gamma$ and the equation $\xi_{2k} y^{s-k} = (\overline{s} - k) y^{\overline{s}-(1-k)}$, for each $s$ with $\Re(s) > 1$, we have
\[
	\xi_{2k} E_{2k}^{(p,q)}(z,s) = (\overline{s} - k) E_{2-2k}^{(p,q)}(z,\overline{s})
\]
and
\[
	\Delta_{2k} E_{2k}^{(p,q)} (z,s) = (s-k)(1-k-s) E_{2k}^{(p,q)} (z,s).
\] 

\subsubsection{The limit formula} 
The following is classically known. 
\begin{proposition}[{\cite[Proposition 6.13]{Iwaniec2002GSM}}]
	The Eisenstein series $E_0^{(p,q)}(z,s)$ of weight $0$ has a meromorphic continuation around $s = 1$ with a simple pole there with residue
	\[
		\Res_{s = 1} E_0^{(p,q)}(z,s) = \frac{\,1\,}{\mathrm{vol}(\Gamma_{p,q} \backslash \bbH)} = \frac{pq}{2\pi r}.
	\]
\end{proposition}

The classical Kronecker limit formula describes the constant term of the Eisenstein series $E_0^{(2,3)}(z,s)$ at $s = 1$. 
Goldstein established a generalization of the limit formula for general Fuchsian groups \cite{Goldstein1973Nagoya}, which yields the following: 
\begin{proposition}[{\cite[(21)]{Goldstein1973Nagoya}}] \label{prop.Kronecker}
	The constant term of the Laurent expansion of $E_0^{(p,q)}(z,s)$ in $s$ at $s=1$, which is also called the limit function, is given by 	
	\begin{align*}
	\mathscr{L}_{p,q}(z)&=\lim_{s \to 1} \big( E_0^{(p,q)} (z,s) - \frac{\,1\,}{\vol(\Gamma_{p,q} \backslash \bbH)} \frac{\,1\,}{s-1} \big) \\
	&=C_{p,q} - \frac{\log y}{\vol(\Gamma_{p,q} \backslash \bbH)}+ \frac{y}{\lambda} + \sum_{n=1}^\infty c_{p,q}(n) q_\lambda^n + \sum_{n=1}^\infty \overline{c_{p,q}(n)} \overline{q_\lambda}^n,
	\end{align*} 
where $C_{p,q}$ and $c_{p,q}(n)$ are complex numbers described in terms of a certain Dirichlet series. 
\end{proposition} 

\subsubsection{A harmonic Maass form of weight 2} 
Let us define a function by $E_2^{(p,q),*}(z) = \xi_0 \mathscr{L}_{p,q}(z)$. Then we have the following. 
\begin{proposition} 
The function $E_2^{(p,q),*}(z)$ is a harmonic Maass form of weight $2$ for $\Gamma_{p,q}$. 
The space $\mathcal{H}_2(\Gamma_{p,q})$ is a $1$-dimensional $\C$-vector space spanned by $E_2^{(p,q),*}(z)$.
\end{proposition}

\begin{proof}
By a direct calculation, we have
\[
	E_2^{(p,q),*} (z) = - \frac{\,1\,}{\vol(\Gamma_{p,q} \backslash \bbH)} \frac{\,1\,}{y} + \frac{\,1\,}{\lambda} + \sum_{n=1}^\infty d_{p,q} (n) q_\lambda^n,
\]
where $d_{p,q} (n) = (-4\pi n/\lambda) c_{p,q} (n)$. Since $\mathscr{L}_{p,q}(z)$ is a $\Gamma_{p,q}$-invariant function, $E_2^{(p,q),*}(z)$ satisfies the modular transformation law of weight $2$. 
The conditions (ii) and (iii) in Definition \ref{def-harmonic-Maass} are easily verified. 
Hence we have $E_2^{(p,q),*}(z) \in \mathcal{H}_2(\Gamma_{p,q})$.

	Let $f \in \mathcal{H}_2(\Gamma_{p,q})$. Since $\mathcal{M}_0(\Gamma_{p,q}) = \C$ by Proposition \ref{valence}, the image $\xi_2 f$ is a constant function. Hence there exists a constant $c \in \C$ such that $\xi_2 (f(z) - c E_2^{(p,q),*}(z)) = 0$, that is, $f(z) - c E_2^{(p,q),*}(z) \in \mathcal{M}_2(\Gamma_{p,q})$. 
	Since $\mathcal{M}_2(\Gamma_{p,q})=0$ by Lemma \ref{lem.M2=0}, we obtain $f(z) -c E_2^{(p,q),*}(z)=0$, completing the proof. 
\end{proof} 

\subsubsection{A mock modular form of weight 2} 
Let $E_2^{(p,q)}(z)$ denote the holomorphic part of the harmonic Maass form $E_2^{(p,q),*}(z)$, so that 
$E_2^{(p,q)}(z)$ is a mock modular form of weight 2 and we have 
\[ 
E_2^{(p,q)}(z) = E_2^{(p,q), *} (z) + \frac{\,1\,}{\mathrm{vol}(\Gamma_{p,q} \backslash \bbH)} \frac{\,1\,}{{\rm Im}(z)} = \frac{\,1\,}{\lambda} + \sum_{n=1}^\infty d_{p,q} (n) q_\lambda^n.
\] 

The modular transformation law of weight 2 for $E_2^{(p,q),*}(z)$ yields the modular gap of the function $E_2^{(p,q)}(z)$ described as follows. 

\begin{lemma} \label{lem.modular-gap} 
	For any $\gamma = \smat{a&b\\c&d} \in \Gamma_{p,q}$, we have
	\begin{align*} 
		(cz+d)^{-2} E_2^{(p,q)} (\gamma z) - E_2^{(p,q)}(z) &= \frac{\,1\,}{\vol(\Gamma_{p,q} \backslash \bbH)} \big( (cz+d)^{-2} \frac{\,1\,}{\Im(\gamma z)} - \frac{\,1\,}{\Im(z)} \big)\\ 
		&= \frac{pq}{r} \frac{c}{\pi i (cz+d)}.
	\end{align*}
\end{lemma} 

This gap will play a crucial role to define a holomorphic cusp form $\Delta_{p,q}(z)$ in the next subsection.


\subsection{The cusp form $\Delta_{p,q}(z)$ of weight {$2pq$}} \label{s2-3} 


In this subsection, we construct a holomorphic cusp form $\Delta_{p,q}(z)$ of weight $2pq$ for $\Gamma_{p,q}$ with no poles and zeros on $\bbH$. 
In the course of argument, we introduce a primitive function $F_{p,q}(z)$, a 1-cocycle function $R_{p,q}(\gamma,z)$, and the Rademacher symbol $\psi_{p,q}:\Gamma_{p,q}\to \Z$ as well. 

\subsubsection{A 1-cocycle function} 
Let $F_{p,q}(z)$ denote the primitive function of $2\pi i r E_2^{(p,q)}(z)$ defined by 
\[ F_{p,q}(z) = \frac{2\pi i rz}{\lambda} + r \lambda \sum_{n=1}^\infty \frac{d_{p,q} (n)}{n} q_\lambda^n = \frac{2\pi i rz}{\lambda} - 4\pi r \sum_{n=1}^\infty c_{p,q} (n) q_\lambda^n,
\] 
where $c_{p,q}(n)$ are those in Proposition \ref{prop.Kronecker}. 
This $F_{p,q}$ is the regularized primitive function in the sense that the leading coefficient of the Fourier expansion of $\Delta_{p,q}(z)=\exp F_{p,q}(z)$ is 1. 

In addition, let  $R_{p,q} : \Gamma_{p,q} \times \bbH \to \C$ denote the weight 0 modular gap function of $F_{p,q}(z)$ defined by 
\[ R_{p,q}(\gamma, z) = F_{p,q}(\gamma z) - F_{p,q}(z), \] 
Then we have $R_{p,q}(-\gamma, z) = R_{p,q}(\gamma, z)$ and the 1-cocycle relation 
\[ R_{p,q}(\gamma_1 \gamma_2, z) = R_{p,q}(\gamma_1, \gamma_2 z) + R_{p,q}(\gamma_2, z). \] 

\subsubsection{The Rademacher symbol $\psi_{p,q}(\gamma)$}
By Lemma \ref{lem.modular-gap}, we have $\dfrac{d}{dz}(R_{p,q}(\gamma, z) - 2pq \log j(\gamma, z))=0$. 
Hence there exists a function $\psi_{p,q} : \Gamma_{p,q} \to \C$ satisfying 
\[ R_{p,q}(\gamma, z) = 2pq \log j(\gamma, z) + 2\pi i \psi_{p,q}(\gamma),\] 
where we assume that $\Im \log j(\gamma, z) \in [-\pi, \pi)$. We call this $\psi_{p,q}$ \emph{the Rademacher symbol for} $\Gamma_{p,q}$. 
Let us verify that $\psi_{p,q}(\gamma)\in \Z$. 
\begin{lemma} \label{psi-generators} 
For the elements $T_{p,q}, S_p, U_q$ of $\Gamma_{p,q}$ $($cf. Definition \ref{Gpq-generator}$)$, we have $\psi_{p,q}(T_{p,q}) = r$, $\psi_{p,q}(S_p) = -q$, and $\psi_{p,q}(U_q) = -p$. 
\end{lemma}

\begin{proof}
Let us first show that $\psi_{p,q}(U_q) = -p$. By the fact that $U_q\!^q = -I$, for any $z\in \bbH$, we have
\[ 0 = R_{p,q}(-I, z) = R_{p,q}(U_q\!^q, z) = \sum_{k=0}^{q-1} R_{p,q}(U_q, U_q\!^k z) = 2pq \sum_{k=0}^{q-1} \log j(U_q, U_q\!^k z) + 2\pi i q \psi_{p,q}(U_q).\]
Since $\displaystyle \sum_{k=0}^{q-1} \log |j(U_q, U_q\!^k z)| = \log |j(-I,z)| = 0$, we see that
\[ \psi_{p,q}(U_q) = -\frac{p}{\pi} \sum_{k=0}^{q-1} \arg j(U_q, U_q\!^k z) = -\frac{p}{\pi} \sum_{k=0}^{q-1} \arg U_q\!^k z, \]
where $\arg z \in [-\pi, \pi)$. Here the left-hand side is independent of the choice of $z$, and the right-hand side is continuous in $z \in \bbH$. By taking the limit $z \to i\infty$ and applying Lemma \ref{Chebyshev}, we obtain $\psi_{p,q}(U_q) = -p$. 
In a similar way, we may obtain $\psi_{p,q}(S_p) = -q$. 
	
Finally, by 
\[ 2\pi i\psi_{p,q}(T_{p,q}) = R_{p,q}(T_{p,q},z) = F_{p,q}(z+\lambda) - F_{p,q}(z) = \dfrac{2\pi ir}{\lambda} ((z+\lambda) - z) = 2\pi ir,\] 
we obtain $\psi_{p,q}(T_{p,q}) = r.$ 
\end{proof}

\begin{proposition} \label{prop.psi-integer-value} 
For any $\gamma \in \Gamma_{p,q}$, the value $\psi_{p,q}(\gamma)$ is an integer.
\end{proposition}

\begin{proof} 
We prove the assertion by induction on the word length of $\gamma \in \Gamma_{p,q}$ with respect to the generators $S_p$ and $U_q$. 
We proved in Lemma \ref{psi-generators} that $\psi_{p,q}(S_p), \psi_{p,q}(U_q) \in \Z$. 
Now suppose that $\psi_{p,q}(\gamma) \in \Z$. If $w \in \{S_p, U_q\}$, then we see that
\begin{align*}
2\pi i \psi_{p,q}(w \gamma) &= R_{p,q}(w \gamma, z) - 2pq \log j(w \gamma, z)\\
&= R_{p,q}(w, \gamma z) + R_{p,q}(\gamma, z) - 2pq \log j(w, \gamma z) j(\gamma, z).
\end{align*}
By $\log j(w, \gamma z) + \log j(\gamma, z) - \log j(w, \gamma z) j(\gamma, z) \in 2\pi i \Z$, we obtain $\psi_{p,q}(w \gamma) \in \Z$. 
\end{proof}

\subsubsection{A cusp form of weight $2pq$}

Finally, we define a holomorphic function on $\bbH$ by  $\Delta_{p,q}(z) = \exp F_{p,q}(z)$.
By Proposition \ref{prop.psi-integer-value}, for any $\gamma \in \Gamma_{p,q}$, we have
\[
	\Delta_{p,q}(\gamma z) = \Delta_{p,q}(z) \exp {R_{p,q}(\gamma, z)} = j(\gamma, z)^{2pq} \Delta_{p,q}(z),
\]
that is, $\Delta_{p,q}|_{2pq} \gamma = \Delta_{p,q}$ holds. By the definition, $\Delta_{p,q}(z)$ is holomorphic, and has no zeros and poles on the upper-half plane $\bbH$. Moreover, by Proposition \ref{valence}, the function vanishes at the cusp $i\infty$. 
Therefore, by the construction, we have the following. 
\begin{proposition} The function $\Delta_{p,q}(z)$ is a cusp form of weight $2pq$ with a unique zero of order $r$ at the cusp $i\infty$, having a Fourier expansion of the form $\Delta_{p,q}(z) = q_\lambda^r + O(q_\lambda^{r+1})$. In addition, we have \[ \frac{d}{dz} \log \Delta_{p,q}(z) = F'_{p,q}(z) = 2\pi i r E_2^{(p,q)}(z). \] 
\end{proposition} 

\begin{remark} For the function $\log \eta_{\Gamma_{p,q}, i}^{} (z/\lambda)$ introduced in \cite[Theorem 3.1]{Goldstein1973Nagoya}, we have $F_{p,q}(z)=4pq \log \eta_{\Gamma_{p,q}, i}^{} (z/\lambda)$ and $\Delta_{p,q}(z) = \eta_{\Gamma_{p,q}, i}^{} (z/\lambda)^{4pq}$. 
However, our $\psi_{p,q}(\gamma)$ and the generalized Dedekind sum $S_{\Gamma_{p,q}, i}^{} (\gamma)$ in \cite{Goldstein1973Nagoya} are slightly different, due to their choices of branches of the logarithm. 
\end{remark}

In terms of our cusp form $\Delta_{p,q}(z)$, the Kronecker limit type formula in Proposition \ref{prop.Kronecker} is paraphrased as follows.
\[ \mathscr{L}_{p,q}(z) = \lim_{s \to 1} \big( E_0^{(p,q)} (z,s) - \frac{\,1\,}{\vol(\Gamma_{p,q} \backslash \bbH)} \frac{\,1\,}{s-1} \big) = -\frac{\,1\,}{\mathrm{vol}(\Gamma_{p,q} \backslash \bbH)} \log (y|\Delta_{p,q}(z)|^{1/pq}) + C_{p,q}. \]

\begin{remark}
	The limit function $\mathscr{L}_{p,q}(z)$ is an example of polyharmonic Maass forms, that were recently introduced by Lagarias--Rhoades in \cite{LagariasRhoades2016} as a generalization of harmonic Maass forms. 
	A real analytic function $f: \bbH \to \C$ is called a polyharmonic Maass form of weight $k \in \Z$ and depth $r \in \Z$ for $\Gamma_{p,q}$ if it satisfies the conditions (i) and (iii) in Definition \ref{def-harmonic-Maass} and (ii)' $(\Delta_k)^r f(z) = 0$. In fact, the function $\mathscr{L}_{p,q}(z)$ satisfies the above three conditions with $k=0$ and $r=2$. For further studies on polyharmonic Maass forms, we refer to \cite{Matsusaka2019} and \cite{Matsusaka2020Ramanujan} written by the first author. 
\end{remark}


\section{The Rademacher symbols} \label{s3}


In the previous section, we introduced the Rademacher symbol $\psi_{p,q}: \Gamma_{p,q} \to \Z$ by using a certain 1-cocycle function $R_{p,q}(\gamma,z)$. 
Let us briefly recall the definition. 
The harmonic Maass form $E_2^{(p,q),*}(z)$ yields the mock modular form $E_2^{(p,q)}(z)$. 
We defined the regularized primitive function $F_{p,q}(z)$ of $2\pi i r E_2^{(p,q)}(z)$ and the cusp form $\Delta_{p,q}(z)$ so that $\Delta_{p,q}(z)=\exp F_{p,q}(z)=q_\lambda^r+O(q_\lambda^{r+1})$ hold for $\lambda=2(\cos\frac{\pi}{p}+\cos \frac{\pi}{q})$. 
We further put $\log \Delta_{p,q}(z)=F_{p,q}(z)$. 
Our symbol $\psi_{p,q}$ may be defined as follows, assuming that $\Im \log z \in [-\pi, \pi)$. 
\begin{definition} \label{def.psi}
The Rademacher symbol $\psi_{p,q}:\Gamma_{p,q}\to \Z$ is a unique function satisfying 
\[ R_{p,q}(\gamma, z) = \log \Delta_{p,q}(\gamma z) - \log \Delta_{p,q}(z) = 2pq \log j(\gamma, z) + 2\pi i \psi_{p,q}(\gamma). \]
\end{definition} 

Since the classical case $\psi_{2,3}$ admits many characterizations as Atiyah and Ghys proved, we may expect that $\psi_{p,q}$ also has many. 
In this section, we establish characterization theorems of $\psi_{p,q}$ from three aspects; cycle integrals of $E_2^{(p,q),*}(z)$, a 2-cocycle $W$ generating the bounded cohomology group $H_b^2(\SL_2 \R; \R)$, and an additive character $\chi_{p,q}:\widetilde{\Gamma}_{p,q} \to \Z$. 
In addition, we introduce several variants $\Phi_{p,q}$, $\Psi_{p,q}$, and $\Psi_{p,q}^{\rm h}$ in a view of the classical cases. 
We obtain several lemmas for our main theorem on the linking number through this section. 

\subsection{Cycle integrals}\label{s3-1}


The group $\Gamma_{p,q}$ acts on $\R\cup\{i\infty\}=\partial \bbH$ via the M\"obius transformation. 
Let $\gamma \in \Gamma_{p,q}$ be a hyperbolic element, that is, $|\tr \gamma|>2$ holds. 
Then, there are exactly two fixed points $w_\gamma, w_\gamma'$ on $\R\subset \R\cup\{i\infty\}$. 
Assume $w_\gamma>w_\gamma'$ and put 
$M_\gamma= \frac{1}{\sqrt{w_\gamma - w'_\gamma}} \smat{w_\gamma & w'_\gamma \\ 1 & 1} \in \SL_2 \R$. 
Then $\gamma$ is diagonalized as 
\[ M_\gamma^{-1} \gamma M_\gamma = \spmx{j(\gamma, w_\gamma) & 0 \\ 0 & j(\gamma, w'_\gamma)} = \spmx{\xi_\gamma & 0 \\ 0 & \xi_\gamma^{-1}}. \] 

Now suppose that $\gamma=\smat{a&b\\c&d} \in \Gamma_{p,q}$ is an element with $a+d>2$ and $c>0$, so that $\xi_\gamma>1$ holds. 
Let $S_\gamma$ denote the geodesic in $\bbH$ connecting two fixed points $w_\gamma$ and $w_\gamma'$. 
Then the action of $\gamma$ preserves the set $S_\gamma$ and sends every point on $S_\gamma$ toward $w_\gamma$. 
The image $\ol{S}_\gamma$ of $S_\gamma$ on the Riemann surface (orbifold) $\Gamma_{p,q}\backslash \bbH$ is a closed geodesic. 
If in addition $\gamma$ is primitive, then the arc on $S_\gamma$ connecting any $z_0 \in S_\gamma$ and $\gamma z_0$ is a lift of the simple closed geodesic $\ol{S}_\gamma$. 
\begin{theorem} \label{thm.cycleintegral}
Let $\gamma=\smat{a&b\\c&d}\in \Gamma_{p,q}$ be a primitive element with $a+d>2$ and $c>0$. Then the cycle integral is given by the Rademacher symbol as 
\[ \int_{\ol{S}_\gamma} E_2^{(p,q),*}(z)dz = \dfrac{1}{r}\psi_{p,q}(\gamma). \] 
\end{theorem} 
\begin{proof} 
For any $z_0 \in S_\gamma$, the cycle integral coincides with the path integral along $S_\gamma$ on $\bbH$ as 
\[ \int_{\ol{S}_\gamma} E_2^{(p,q),*}(z)dz = \int_{z_0}^{\gamma z_0} E_2^{(p,q),*}(z)dz. \] 
We let $z_0 = M_\gamma i$. Recall that the harmonic Maass form $E_2^{(p,q),*}(z)$ may be written the sum of holomorphic and non-holomorphic parts as 
\[ E_2^{(p,q),*}(z) = E_2^{(p,q)}(z) - \frac{\,1\,}{\mathrm{vol}(\Gamma_{p,q} \backslash \bbH)} \frac{\,1\,}{{\rm Im}(z)}. \]
The integration of the holomorphic part is given by 
\[ \int_{z_0}^{\gamma z_0} E_2^{(p,q)} (z) dz = \frac{\,1\,}{2\pi i r} \bigg[ F_{p,q}(z) \bigg]_{z_0}^{\gamma z_0} = \frac{\,1\,}{2\pi i r} R_{p,q}(\gamma, z_0). \] 
As for the integration of the non-holomorphic part, recall that ${\rm vol}(\Gamma_{p,q} \backslash \bbH)=\frac{2\pi r}{pq}$. 
In addition, by changing variables via $z = M_\gamma iy$, we obtain 
\begin{align*} 
&\int_{z_0}^{\gamma z_0} \frac{dz}{\Im(z)} 
= \int_1^{\xi_\gamma^2} \frac{\,1\,}{{\Im}(M_\gamma iy)} \frac{i dy}{j(M_\gamma, iy)^2} 
= \int_1^{\xi_\gamma^2} \frac{1-iy}{1+iy} \frac{i dy}{y}\\ 
&= \int_1^{\xi_\gamma^2} \big(\frac{\,i\,}{y} + \frac{2}{1+iy} \big) dy 
= 2i \log \frac{\xi_\gamma j(M_\gamma, i)}{j(M_\gamma, i\xi_\gamma^2)} 
= 2i \log \frac{j(M_\gamma, i)}{j(M_\gamma \spmx{\xi_\gamma & 0 \\ 0 & \xi_\gamma^{-1}}, i)}\\ 
&= 2i \log \frac{j(M_\gamma, i)}{j(\gamma M_\gamma, i)} 
= -2i \log j(\gamma, M_\gamma i), 
\end{align*} 
where we assume that $\Im \log z \in [-\pi, \pi)$ as before. 
By summation, we have 
\[ \int_{z_0}^{\gamma z_0} E_2^{(p,q),*}(z) dz 
= \frac{\,1\,}{2\pi i r} \big(R_{p,q}(\gamma, z_0) - 2pq \log j(\gamma, z_0) \big) 
= \frac{\,1\,}{r} \psi_{p,q}(\gamma), \] 
which finishes the proof.
\end{proof}


\subsection{The 2-cocycle $W$} 
\label{s3-2}


In this subsection, we give an alternative definition of the Rademacher symbol $\psi_{p,q}$ without use of automorphic forms. 
We introduce a bounded 2-cocycle $W$ following Asai \cite{Asai1970NMJ37} and 
prove that $\psi_{p,q}$ is a unique function satisfying $2pqW=-\delta^1\psi_{p,q}$.

\subsubsection{A 2-cocycle and Asai's sign function} 
Here we introduce a 2-cocycle $W$ corresponding to the universal covering group $\wt{\SL_2}\R$ together with an explicit description with use of Asai's sign function. 

\begin{definition} \label{def.W} 
We define a 2-cocycle $W:\SL_2\R\times \SL_2\R\to \Z$ by 
\[W(\gamma_1, \gamma_2) = \frac{\,1\,}{2\pi i} \big( \log j(\gamma_1, \gamma_2 z) + \log j(\gamma_2, z) - \log j(\gamma_1 \gamma_2, z) \big),\]
assuming $\arg j(\gamma, z) = \Im \log j(\gamma, z) \in [-\pi, \pi)$.
\end{definition} 
This is equivalent to say that we have $j(\gamma_1, \gamma_2 z) j(\gamma_2, z) = j(\gamma_1 \gamma_2, z) e^{2\pi i W(\gamma_1, \gamma_2)}$ in  the universal covering group $\widetilde{\C^\times}$ of the multiplicative group $\C^\times = \C - \{0\}$. 
Since the right-hand side of the definition is continuous in $z$, the value of $W$ is independent of $z$. 
We may easily verify the 2-cocycle condition 
\[ W(\gamma_1 \gamma_2, \gamma_3) + W(\gamma_1, \gamma_2) = W(\gamma_1, \gamma_2 \gamma_3) + W(\gamma_2, \gamma_3). \] 

The universal cover $\wt{\SL_2}\R \to \SL_2\R$ as manifolds is a group homomorphism as well. The group $\wt{\SL_2}\R$ is called the universal covering group of $\SL_2\R$. 
Note in addition that each central extension of $\SL_2\R$ by $\Z$ corresponds to a 2-cocycle $\SL_2\R\times \SL_2\R \to \Z$ and each isomorphism class of central extensions corresponds to a 2nd cohomology class in $H^2(\SL_2\R;\Z)$ (cf.~\cite[Chapter IV]{Brown}). 
Now we have the following. 

\begin{proposition} 
As a group, the universal covering group $\wt{\SL_2}\R$ of $\SL_2\R$ is a central extension of $\SL_2 \R$ by $\Z$ corresponding to the $2$-cocycle $W$. In other words, when we identify $\wt{\SL_2}\R$ with $\SL_2 \R \times \Z$ as sets, we have
\[ (\gamma_1, n_1) \cdot (\gamma_2, n_2) = (\gamma_1 \gamma_2, n_1 + n_2 + W(\gamma_1, \gamma_2)) \] 
for every $(\gamma_1, n_1), (\gamma_2, n_2) \in \wt{\SL_2}\R$. 
\end{proposition}

By virtue of the convention $\arg j(\gamma, z) = \Im \log j(\gamma, z) \in [-\pi, \pi)$, we have $W(\gamma_1, \gamma_2) \in \{-1, 0, 1\}$. 
Asai introduced the following sign function to explicitly express the values of $W$. 

\begin{definition}
	For any $\gamma = \smat{a&b\\c&d} \in \SL_2 \R$, we define its \emph{sign} by
	\begin{align*}
		\sgn(\gamma) &= \begin{cases}
			\sgn c &\text{if } c \neq 0,\\
			\sgn a = \sgn d &\text{if } c = 0
		\end{cases}\\
		&= \begin{cases}
			1 &\text{if } 0 \leq \arg j(\gamma, z) < \pi,\\
			-1 &\text{if } -\pi \leq \arg j(\gamma, z) < 0.
		\end{cases}
	\end{align*}
\end{definition}

\begin{proposition}[{\cite{Asai1970NMJ37}}]  
The values of $W(\gamma_1, \gamma_2)$ are given by the following table. 
\begin{center}
  \begin{tabular}{|c|c|c||c|} \hline
    $\sgn(\gamma_1)$ & $\sgn(\gamma_2)$ & $\sgn(\gamma_1\gamma_2)$ & $W(\gamma_1, \gamma_2)$ \\ \hline \hline
    $1$ & $1$ & $-1$ & $1$ \\
    $-1$ & $-1$ & $1$ & $-1$ \\
     & otherwise &  & $0$ \\ \hline
  \end{tabular}
\end{center} 
\end{proposition}  

\begin{remark} 
Asai showed that there is no function $V: \SL_2 \R \to \R$ satisfying 
\[
	W(\gamma_1, \gamma_2) = -(\delta^1 V) (\gamma_1, \gamma_2) = V(\gamma_1 \gamma_2) - V(\gamma_1) - V(\gamma_2), 
\]
that is, the cohomology class $[W]$ in $H^2(\SL_2 \R; \R)$ is non-trivial \cite[Theorem 1]{Asai1970NMJ37}. 

Note that $W$ is a bounded $2$-cocycle by $|W(\gamma_1, \gamma_2)| \leq 1$. 
The bounded cohomology group $H_b^2(\SL_2 \R; \R)\cong \R$ naturally injects into $H^2(\SL_2 \R; \R)$, 
hence the class of $W$ in $H_b^2(\SL_2 \R; \R)$ is a generator 
(cf.~\cite{MatsumotoMorita1985, BargeGhys1992, Frigerio2017book}). 
\end{remark}

\subsubsection{A 2-coboundary of $\Gamma_{p,q}$}
We next calculate the cohomology of $\Gamma_{p,q}$ and establish the relation between the 2-cocycle $W$ on $\Gamma_{p,q}$ and the Rademacher symbol $\psi_{p,q}$. 

\begin{lemma} 
We have $H^1(\Gamma_{p,q}; \Z) = H^1(\Gamma_{p,q}; \C) = \{0\}$ and $H^2(\Gamma_{p,q}; \Z) \cong \Z/2pq \Z$. 
\end{lemma}

\begin{proof}
	Since the triangle group $\Gamma_{p,q}$ is generated by torsion elements $S_p$ and $U_q$, a group homomorphism $f: \Gamma_{p,q} \to \C$ is trivial. Hence $H^1(\Gamma_{p,q};\Z)=H^1(\Gamma_{p,q}; \C) = \{0\}$. 
	
The second assertion follows from the facts 
\[
	H^i(\Z/n\Z; \Z) \cong \begin{cases}
		\Z &\text{if } i=0,\\
		\Z/n\Z &\text{if } i \text{ is even with }i> 0,\\
		0 &\text{if } i \text{ is odd},
	\end{cases}
\]
$\Gamma_{p,q} \cong \Z/2p\Z *_{\Z/2\Z} \Z/2q\Z$, and the Mayer--Vietoris sequence for group cohomology.
\end{proof}

Since $H^2(\Gamma_{p,q}; \Z) \cong \Z/2pq \Z$, we have $2pq[W]=0$ in $H^2(\Gamma_{p,q}; \Z)$. 
Hence there exists a function $f: \Gamma_{p,q} \to \Z$ satisfying the coboundary condition 
\[ 2pq W(\gamma_1, \gamma_2) = -(\delta^1 f)(\gamma_1, \gamma_2) = f(\gamma_1 \gamma_2) - f(\gamma_1) - f(\gamma_2) \] 
for every $\gamma_1, \gamma_2 \in \Gamma_{p,q}$. 
Such $f$ is unique. Indeed, if there are two functions $f_1, f_2 : \Gamma_{p,q} \to \C$ satisfying the same coboundary condition, then the difference $f_1 - f_2$ is a homomorphism. Hence by $H^1(\Gamma_{p,q}; \C) = \{0\}$, we have  $f_1-f_2=0$. 
We further have the following. 
\begin{theorem} \label{thm.W} 
The Rademacher symbol $\psi_{p,q}:\Gamma_{p,q}\to \Z$ 
is a unique function satisfying 
\[ 2pqW|_{\Gamma_{p,q}}=-\delta^1\psi_{p,q}.\]
\end{theorem}

\begin{proof} It suffices to verify that the equality 
\[ \psi_{p,q}(\gamma_1 \gamma_2) - \psi_{p,q}(\gamma_1) - \psi_{p,q}(\gamma_2) = 2pq W(\gamma_1, \gamma_2) \]
 holds for every $\gamma_1, \gamma_2 \in \Gamma_{p,q}$. 
 By the definition of $\psi_{p,q}$, the left-hand side equals
\begin{align*}
\frac{\,1\,}{2\pi i} \big( &R_{p,q}(\gamma_1 \gamma_2, z) - 2pq \log j(\gamma_1 \gamma_2, z) - R_{p,q}(\gamma_1, \gamma_2 z) + 2pq \log j(\gamma_1, \gamma_2 z)\\
& - R_{p,q}(\gamma_2, z) + 2pq \log j(\gamma_2, z) \big).
\end{align*}
The $R_{p,q}$-terms cancel out by the 1-cocycle relation and the remaining equals $2pq W(\gamma_1, \gamma_2)$. 
\end{proof} 


\subsection{The additive character $\chi_{p,q} : \widetilde{\Gamma}_{p,q} \to \Z$} \label{s3-3} 

In this subsection, we provide another characterization of the Rademacher symbol $\psi_{p,q}$ by using an additive character $\chi_{p,q} : \widetilde{\Gamma}_{p,q} \to \Z$.  

As before, we assume that $\wt{\SL_2}\R=\SL_2\R\times \Z$ as a set. 
Let $P:\wt{\SL_2}\R\to \SL_2\R; (\gamma,n)\mapsto \gamma$ denote the universal covering map and put $\wt{\Gamma}_{p,q}=P^{-1}(\Gamma_{p,q})$, so that we have 
\[\wt{\Gamma}_{p,q}=\{(\gamma,n)\in \wt{\SL_2}\R \mid \gamma \in \Gamma_{p,q}, n \in \Z\}. \] 
For each $\gamma \in \Gamma_{p,q}$, we define the standard lift by $\wt{\gamma}=(\gamma,0) \in \wt{\Gamma}_{p,q}$. 

\begin{lemma} The lifts of $S_p,U_q,T_{p,q} \in \Gamma_{p,q}$ satisfy 
\[ \widetilde{S}_p\!^p = (-I, 1) = \widetilde{U}_q\!^q, \qquad \widetilde{T}_{p,q} = \widetilde{-I} \widetilde{U}_q \widetilde{S}_p. \]
The group $\widetilde{\Gamma}_{p,q}$ is generated by $\widetilde{S}_p$ and $\widetilde{U}_q$.
\end{lemma}

\begin{proof} The equalities immediately follow from the group operation of $\wt{\SL_2}\R$ with use of $W$ and Lemma \ref{Chebyshev}. 
Since we have $(I,1) = \widetilde{S}_p\!^{2p} = \widetilde{U}_q\!^{2q}$, 
the elements $\widetilde{S}_p$ and $\widetilde{U}_q$ generate $\widetilde{\Gamma}_{p,q}$.
\end{proof}

Let $\chi: \widetilde{\Gamma}_{p,q} \to \Z$ be an additive character, that is, a group homomorphism to the additive group $\Z$. Such $\chi$ is determined by the values $\chi(\widetilde{S}_p) = s$ and $\chi(\widetilde{U}_q) = u$. 
The relation $\widetilde{S}_p\!^p = (-I, 1) = \widetilde{U}_q\!^q$ imposes the condition $\chi(-I, 1) = ps = qu$ on the pair $(s,u)$. Since $p$ and $q$ are coprime, we have $s = mq, u = mp$ for some $m\in \Z$. 
In addition, since $(-I,1)^2 = (I,1)$, we have $\chi(I,1) = 2mpq$. 

Define a function $V:\Gamma_{p,q}\to \Z$ by putting $V(\gamma)=\chi(\wt{\gamma})$. 
Then we have $\chi(\gamma, n) = \chi(\widetilde{\gamma} \cdot (I,1)^n) = V(\gamma) + 2mnpq$ for any $(\gamma,n)\in \wt{\Gamma}_{p,q}$. In addition, for any $\gamma_1, \gamma_2\in \Gamma_{p,q}$, 
by the relation $\widetilde{\gamma}_1 \cdot \widetilde{\gamma}_2 = (\gamma_1 \gamma_2, W(\gamma_1, \gamma_2))$, 
we have 
\[
V(\gamma_1 \gamma_2) = V(\gamma_1) + V(\gamma_2) - 2mpq W(\gamma_1, \gamma_2).
\]

If $m=-1$, then Theorem \ref{thm.W} yields $V = \psi_{p,q}$. Consequently, we obtain the following.

\begin{theorem} \label{thm.chi} 
The additive character $\chi_{p,q}: \widetilde{\Gamma}_{p,q} \to \Z$ determined by $\chi_{p,q}(\widetilde{S}_p) = -q$ and $\chi_{p,q}(\widetilde{U}_q) = -p$ satisfies 
\[ \psi_{p,q}(\gamma) = \chi_{p,q}(\gamma, n) + 2npq \]
for every $\gamma \in \Gamma_{p,q}$ and $n \in \Z$.
\end{theorem}

\begin{remark} Theorem \ref{thm.chi} is a generalization of Asai's result in his unpublished lecture note \cite{Asai2003Shizuoka}. 
His function $\Phi$ satisfies $\Phi(\gamma)=\psi_{2,3}(\gamma)+3\sgn(\gamma)$ for any $\gamma\in \SL_2\Z$. 
\end{remark} 

For the convenience of later use, let us calculate the values of the Rademacher symbol at several elements. 
By Theorem \ref{thm.chi}, we easily see 
\begin{align*}
	\psi_{p,q}(-I) &= \chi_{pq}(-I,1) + 2pq = -pq + 2pq = pq,\\
	\psi_{p,q}(T_{p,q}) &= \chi_{p,q}(\widetilde{-I} \widetilde{U}_q \widetilde{S}_p) = pq - p - q = r. 
\end{align*} 
The latter agrees with the previous result in Lemma \ref{psi-generators}. In addition, we have the following. 

\begin{lemma}\label{lem.psi-minus-gamma}
	For any $\gamma = \smat{a&b\\c&d} \in \Gamma_{p,q}$, we have
	\begin{align*}
		\psi_{p,q}(-\gamma) &= \psi_{p,q}(\gamma) + pq \sgn(\gamma),\\
		\psi_{p,q}(\gamma^{-1}) &= \begin{cases}
			-\psi_{p,q}(\gamma) + 2pq &\text{if } c = 0, d < 0,\\
			-\psi_{p,q}(\gamma) &\text{if\ otherwise}.
		\end{cases}
	\end{align*}
\end{lemma}

\begin{proof}
	By Theorem \ref{thm.W}, we have 
	\[
		\psi_{p,q}(-\gamma) = \psi_{p,q}(\gamma) + \psi_{p,q}(-I) + 2pq W(-I, \gamma).
	\]
	Recall $\psi_{p,q}(-I) = pq$. Since $W(-I, \gamma) = 0$ if $\sgn(\gamma) = +1$ and $W(-I, \gamma) = -1$ if $\sgn(\gamma) = -1$, we have $\psi_{p,q}(-I) + 2pq W(-I, \gamma) = pq \sgn(\gamma)$.
	
	In general, the inverse of any $(\gamma, n) \in \widetilde{\SL_2}\R$ is given by
	\[
		(\gamma, n)^{-1} = (\gamma^{-1}, -n - W(\gamma, \gamma^{-1})) = \begin{cases}
		(\gamma^{-1}, -n+1) &\text{if } c=0, d<0,\\
		(\gamma^{-1}, -n) &\text{if\ otherwise}.
	\end{cases}
	\]
	Hence we have
	\[
		\psi_{p,q}(\gamma^{-1}) = \chi_{p,q}(\gamma^{-1}, 0) = \chi_{p,q}((\gamma, 1)^{-1}) = -\chi_{p,q}(\gamma, 1) = -\psi_{p,q}(\gamma) +2pq
	\]
	if $c = 0, d < 0$, and 
	\[
		\psi_{p,q}(\gamma^{-1}) = \chi_{p,q}((\gamma,0)^{-1}) = -\psi_{p,q}(\gamma)
	\]
	if otherwise. 
\end{proof}

We also use the following lemma later. 

\begin{lemma}\label{lem.psi=1}
Let $(x,y) \in \Z^2$ be a pair satifying $px+qy=1$, $|x| < q,$ $|y| < p,$ $xy<0$ and put  
$\gamma = U_q\!^{-x} S_p\!^{-y} = \smat{a&b\\c&d} \in \Gamma_{p,q}$. Then we have
\begin{itemize}
\item[$\bullet$] $\gamma = T_{2,3} = \smat{1 & 1 \\ 0 & 1}$ if $(p,q) = (2,3)$,
\item[$\bullet$] $\tr \gamma> 2$ and $c > 0$ if $(p,q) \neq (2,3)$.
\end{itemize}
In both cases, we have $\psi_{p,q}(\gamma) = 1$.
\end{lemma}

\begin{proof}
If $(p,q) = (2,3)$, then we have $\gamma = U_3^{-2} S_2 = U_3 S_2^{-1} = \smat{1 & 1 \\ 0 & 1}$. 
	
If $(p,q) \neq (2,3)$, then by Lemma~\ref{Chebyshev}, we have 
\[ \gamma = \frac{\,1\,}{\sin \frac{\pi}{p}} \frac{\,1\,}{\sin \frac{\pi}{q}} \pmat{-\sin \frac{\pi (x-1)}{q} & \sin \frac{\pi x}{q}\\ -\sin \frac{\pi x}{q} & \sin \frac{\pi(x+1)}{q}} \pmat{\sin \frac{\pi(y+1)}{p} & \sin \frac{\pi y}{p} \\ -\sin \frac{\pi y}{p} & -\sin \frac{\pi(y-1)}{p}}. \]
By the condition $xy < 0$, we have $c > 0$. In addition, we have 
\begin{align*}
\tr \gamma &= - \frac{2}{ \sin \frac{\pi}{p} \cdot \sin \frac{\pi}{q}} \big( \sin \frac{\pi x}{q} \cos \frac{\pi}{q} \sin \frac{\pi y}{p} \cos \frac{\pi}{p} - \sin \frac{\pi}{q} \cos \frac{\pi x}{q} \sin \frac{\pi}{p} \cos \frac{\pi y}{p} + \sin \frac{\pi x}{q} \sin \frac{\pi y}{p} \big)\\
&= 2 \big(\frac{\left|\sin \frac{\pi x}{q} \sin \frac{\pi y}{p}\right|}{\sin \frac{\pi}{p} \sin \frac{\pi}{q}} \big(\cos \frac{\pi}{p} \cos \frac{\pi}{q} +1 \big) + \cos \frac{\pi x}{q} \cos \frac{\pi y}{p} \big) >2. 
\end{align*} 
	
For any $m,n$ with $0 < |m| < p$ and $0 < |n| < q$, we have $\sgn(S_p\!^m) = \sgn(m)$ and $\sgn(U_q\!^n) = \sgn(n)$. Hence we have $\psi_{p,q}(U_q\!^{-x}) = -x \psi_{p,q}(U_q) = px$ and $\psi_{p,q}(S_p\!^{-y}) = -y \psi_{p,q}(S_p) = qy$. By $xy < 0$, we obtain
\[ \psi_{p,q}(\gamma) = \psi_{p,q}(U_q\!^{-x}) + \psi_{p,q}(S_p\!^{-y}) + 2pq W(U_q\!^{-x}, S_p\!^{-y}) = px+qy = 1. 
\qedhere \]
\end{proof}


\subsection{Class-invariant functions} 
\label{subsec.variants} 

In this subsection, we recall several variants of the classical Rademacher symbol and generalize them for any $\Gamma_{p,q}$. 
We modify the Rademacher symbol $\psi_{p,q}$ to obtain a class-invariant function, namely, the original Rademacher symbol $\Psi_{p,q}$. In addition, we define the Dedekind symbol $\Phi_{p,q}$ and the homogeneous Rademacher symbol $\Psi_{p,q}^{\rm h}$ and attach remarks. 

\subsubsection{The classical cases} 

Let us recollect two classical variant $\Phi_{2,3}$ and $\Psi_{2,3}$ of the Rademacher symbol $\psi_{2,3}$. 
The Dedekind symbol $\Phi_{2,3}: \SL_2 \Z \to \Z$ introduced by Dedekind in 1892 \cite{Dedekind1892} is defined as a unique function satisfying  
\[
 \log \Delta_{2,3} \left(\gamma z \right) - \log \Delta_{2,3} (z) = 
\begin{cases}
\displaystyle{12 \log \big(\frac{cz+d}{i \sgn c} \big) + 2\pi i \Phi_{2,3}(\gamma)} &\text{if } c \neq 0,\\
2\pi i \Phi_{2,3}(\gamma) &\text{if } c = 0,
\end{cases}
\]
for every $\gamma=\smat{a&b\\c&d}\in \SL_2\Z$ and $z\in \bbH$, 
assuming $\Im \log z \in (-\pi, \pi)$. Here, $\sgn c \in \{-1, 0, 1\}$ denotes the usual sign function. 
For each $a\in \Z$ and $c\in \Z_{>0}$, the Dedekind sum is defined by 
\[s(a,c) = \sum_{k=1}^{c-1} (\!(\dfrac{k}{c})\!) (\!(\dfrac{ka}{c})\!),\] 
where we put $(\!(x)\!) = x - \lfloor x \rfloor - 1/2$ if $x \not\in \Z$ and $(\!(x)\!) = 0$ if $x \in \Z$. 
The following formula is due to Dedekind: 
\[
\Phi_{2,3} (\spmx{a&b\\c&d} ) = 
\begin{cases}
\dfrac{a+d}{c} - 12 \sgn c \cdot s(a, |c|) \qquad &\text{if } c \neq 0,\\
\dfrac{\,b\,}{d} &\text{if } c=0.
\end{cases}
\]

This symbol $\Phi_{2,3}$ is not a class-invariant function. 
In 1956 \cite{Rademacher1956}, Rademacher introduced a class-invariant function by modifying the Dedekind symbol, namely, he defined the original Rademacher symbol $\Psi_{2,3}:\SL_2\Z\to \Z$ by putting 
\[ \Psi_{2,3}(\gamma) = \Phi_{2,3}(\gamma) -3 \sgn(c(a+d)). \] 
This symbol $\Psi_{2,3}$ satisfies
\[ \Psi_{2,3}(\gamma) = \Psi_{2,3}(-\gamma) = -\Psi_{2,3}(\gamma^{-1}) = \Psi_{2,3}(g^{-1} \gamma g) \]
for any $\gamma, g \in \SL_2 \Z$. In addition, if $\tr \gamma  >0$, then
\[ \log \Delta_{2,3}(\gamma z) - \log \Delta_{2,3}(z) = 12 \log j(\gamma, z) + 2\pi i \Psi_{2,3}(\gamma) \]
holds, that is, we have $\Psi_{2,3}(\gamma)=\psi_{2,3}(\gamma)$.

We remark that there are many more variants in literatures with confusions. 
The clarification between $\Phi_{2,3}$ and $\Psi_{2,3}$ is due to \cite{DukeImamogluToth2017Duke}. 

\subsubsection{The original symbol $\Psi_{p,q}$} 


Let us generalize the original symbol for any $\Gamma_{p,q}$. 
\begin{definition} \label{def.original} 
We define \emph{the original Rademacher symbol $\Psi_{p,q}:\Gamma_{p,q}\to \Z$ for} $\Gamma_{p,q}$ by
\[ \Psi_{p,q}(\gamma) = \psi_{p,q}(\gamma) + \frac{pq}{2} \sgn(\gamma) (1 - \sgn \tr \gamma ), \]
where $\sgn(\gamma)\in \{\pm1\}$ denotes Asai's sign function and $\sgn \tr \gamma  \in \{-1, 0, 1\}$ the usual sign function. 
\end{definition}
If we put $(p,q)=(2,3)$, then we obtain the classical symbol $\Psi_{2,3}$ due to Rademacher. 
If $\tr \gamma>0$, then $\Psi_{p,q}(\gamma)=\psi_{p,q}(\gamma)$ holds. 
The following assertion is proved by Lemmas \ref{lem.-gamma}--\ref{lem.conjugate}.
\begin{proposition} \label{prop.class-invariant} 
For any $\gamma, g \in \Gamma_{p,q}$, 
\[ \Psi_{p,q}(\gamma) = \Psi_{p,q}(-\gamma) = - \Psi_{p,q}(\gamma^{-1}) = \Psi_{p,q}(g^{-1} \gamma g) \]
holds. 
In addition, if $|\tr \gamma | \geq 2$, then $\Psi_{p,q}(\gamma^n) = n \Psi_{p,q}(\gamma)$ holds for any $n \in \Z$.
\end{proposition} 

\begin{lemma} \label{lem.-gamma}
For any $\gamma \in \Gamma_{p,q}$, we have $\Psi_{p,q}(-\gamma) = \Psi_{p,q}(\gamma)$, that is, $\Psi_{p,q}$ induces a function on $\Gamma_{p,q} / \{\pm I\}$.
\end{lemma}

\begin{proof}
By Lemma \ref{lem.psi-minus-gamma}, we obtain
\begin{equation*} \begin{aligned} 
\Psi_{p,q}(-\gamma) &= \psi_{p,q}(-\gamma) + \frac{pq}{2} \sgn(-\gamma) (1 - \sgn \tr(-\gamma) )\\
&= \psi_{p,q}(\gamma) + pq \sgn(\gamma) - \frac{pq}{2} \sgn(\gamma) (1 + \sgn \tr \gamma ) =\Psi_{p,q}(\gamma).\\[-6mm] 
\end{aligned} \end{equation*} 
\end{proof}

\begin{lemma}
For any $\gamma \in \Gamma_{p,q}$, we have $\Psi_{p,q}(\gamma^{-1}) = -\Psi_{p,q}(\gamma)$.
\end{lemma}

\begin{proof}
If $\gamma = \smat{a & b \\c & d} \in \Gamma_{p,q}$ satisfies $c = 0$ and $d <0$, then by Lemma \ref{lem.psi-minus-gamma}, 
\[ \Psi_{p,q}(\gamma^{-1}) = \psi_{p,q}(\gamma^{-1}) - pq = -\psi_{p,q}(\gamma) + pq = -\Psi_{p,q}(\gamma). \]
Other cases are obtained in as similar manner. 
\end{proof}

\begin{lemma} 
For $\gamma \in \Gamma_{p,q}$ with $|\tr \gamma | \geq 2$, we heve $\Psi_{p,q}(\gamma^n) = n \Psi_{p,q}(\gamma)$.
\end{lemma}

\begin{proof}
Since $-\Psi_{p,q}(\gamma^{-1}) = \Psi_{p,q}(-\gamma) = \Psi_{p,q}(\gamma)$ holds by the above lemmas, we may assume $\sgn(\gamma) > 0$, $\tr \gamma  \geq 2$, and $n > 0$ without loss of generality. Put $t = \tr \gamma  \geq 2$. Then we have $\gamma^n = a_n(t) \gamma - a_{n-1}(t) I$, where $a_0(t) = 0, a_1(t) = 1$, and $a_n(t) = t a_{n-1}(t) - a_{n-2}(t)$. This implies that $\sgn(\gamma^n) > 0$ and $\tr(\gamma^n) > 0$ for any $n>0$. Hence we obtain
\[ \Psi_{p,q}(\gamma^n) = \psi_{p,q} (\gamma^n) = \chi_{p,q} (\gamma^n, 0) = \chi_{p,q}((\gamma,0)^n) = n \psi_{p,q}(\gamma) = n \Psi_{p,q}(\gamma), \]
which conclude the proof.
\end{proof}

\begin{lemma} \label{lem.conjugate}
The function $\Psi_{p,q}(\gamma)$ is a class-invariant function, that is, for any $g \in \Gamma_{p,q}$, we have $\Psi_{p,q}(g^{-1} \gamma g) = \Psi_{p,q}(\gamma)$.
\end{lemma}

\begin{proof}
We may assume $\sgn(\gamma) > 0$ and $\tr \gamma  \geq 0$ without loss of generality. It suffices to show the equation $\Psi_{p,q}(g^{-1} \gamma g) = \Psi_{p,q}(\gamma)$ for generators $g = T_{p,q}, S_p$. By the definitions, we have
\begin{align*}
\Psi_{p,q} (g^{-1} \gamma g) &= \Psi_{p,q}(g^{-1}) + \Psi_{p,q}(\gamma) + \Psi_{p,q}(g) + 2pq ( W(g^{-1}, \gamma g) + W(\gamma, g) )\\
&\qquad + \frac{pq}{2} \Big( \sgn(g^{-1} \gamma g) (1- \sgn \tr(g^{-1} \gamma g) ) - \sgn(g^{-1}) (1 - \sgn \tr(g^{-1}))\\
&\qquad \qquad - 1 + \sgn \tr \gamma  - \sgn(g) (1 - \sgn \tr g ) \Big).
\end{align*}
By $\Psi_{p,q}(g^{-1}) + \Psi_{p,q}(g) = 0$ and 
\[ W(\gamma_1, \gamma_2) = \frac{\,1\,}{4} \Big(\sgn(\gamma_1) + \sgn(\gamma_2) - \sgn(\gamma_1 \gamma_2) - \sgn(\gamma_1) \sgn(\gamma_2) \sgn(\gamma_1 \gamma_2) \Big), \]
we obtain
\begin{align*}
\Psi_{p,q}(g^{-1} \gamma g) = &\Psi_{p,q}(\gamma) + \frac{pq}{2} \Big(- \sgn(g^{-1}) \sgn(\gamma g) \sgn(g^{-1} \gamma g) - \sgn(g) \sgn(\gamma g)\\
&- \sgn(g^{-1} \gamma g) \sgn \tr \gamma  + \sgn(g^{-1}) \sgn \tr(g^{-1}) + \sgn \tr \gamma  + \sgn(g) \sgn \tr g \Big).
\end{align*}

If $g = T_{p,q}$, then we have $\sgn(\gamma g) = \sgn(g^{-1} \gamma g) = \sgn(\gamma) = 1$, that is, $\Psi_{p,q}(g^{-1} \gamma g) = \Psi_{p,q}(\gamma)$. 

If $g = S_p$, then we have
\begin{align*}
\Psi_{p,q}(g^{-1} \gamma g) = &\Psi_{p,q}(\gamma) + \frac{pq}{2} \big(\sgn(\gamma g) - \sgn \tr \gamma  \big) \big(\sgn(g^{-1} \gamma g)-1\big).
\end{align*} 
Assume $\gamma = \smat{a&b\\c&d}$ with $a+d \geq 0$ and $\sgn(\gamma) > 0$. Then we see that
\[ \gamma S_p = \pmat{b & -a + 2b\cos \frac{\pi}{p} \\ d & -c + 2d\cos \frac{\pi}{p}}, \qquad S_p\!^{-1} \gamma S_p = \pmat{d + 2b \cos \frac{\pi}{p} & * \\ -b & a - 2b\cos \frac{\pi}{p}}.\]
\begin{enumerate}
		\item If $a + d = 0$, then $-bc = a^2 + 1 > 0$. Thus we have $c > 0$ and $-b > 0$, that is, $\sgn(g^{-1} \gamma g) = 1$.
		\item If $a + d > 0$, then it suffices to show that $(\sgn(\gamma S_p) - 1)(\sgn(S_p^{-1} \gamma S_p) - 1) = 0$. 
			\begin{itemize}
				\item If $d > 0$, then we have $\sgn(\gamma S_p) = 1$.  
				\item If $d = 0$, then $\sgn(\gamma S_p) = \sgn(-c) < 0$. In addition, by $\det(\gamma) = -bc = 1$, we have $b < 0$. Hence we obtain $\sgn(S_p\!^{-1} \gamma S_p) = 1$.
				\item If $d < 0$, then $\sgn(\gamma S_p) = -1$. In this case, we have $a > 0, c >0$, and $b < 0$. Hence we have $\mathrm{sgn}(S_p\!^{-1} \gamma S_p) = 1$.
			\end{itemize}
	\end{enumerate}
	In conclusion, we obtain $\Psi_{p,q}(g^{-1} \gamma g) = \Psi_{p,q}(\gamma)$ for all cases.
\end{proof}

\subsubsection{Other variants $\Phi_{p,q}$ and $\Psi_{p,q}^{\rm h}$}
Here, we discuss two more variants $\Phi_{p,q}$ and $\Psi_{p,q}^{\rm h}$. 
\begin{definition}
We define \emph{the Dedekind symbol} $\Phi_{p,q} : \Gamma_{p,q} \to \frac{\,1\,}{2} \Z$ by
\[ \Phi_{p,q}(\gamma) = \Psi_{p,q}(\gamma) + \frac{pq}{2} \sgn(c(a+d)). \]
\end{definition} 
This symbol $\Phi_{p,q}$ is a unique function satisfying
\[ \Phi_{p,q}(\gamma_1 \gamma_2) - \Phi_{p,q}(\gamma_1) - \Phi_{p,q}(\gamma_2) = - \frac{pq}{2} \sgn(c_1 c_2 c_{12}) \]
for every $\gamma_i = \smat{* & * \\ c_i & *} \in \Gamma_{p,q}$ with $\gamma_1 \gamma_2 = \smat{* & * \\ c_{12} & *}$, hence a generalization of \cite[(62)]{RademacherGrosswald1972}. The values at generators are given by
\[ \Phi_{p,q}(T_{p,q}) = r = pq-p-q, \quad \Phi_{p,q}(S_p) = \frac{q(p-2)}{2}, \quad \Phi_{p,q}(U_q) = \frac{p(q-2)}{2}. \]
\begin{definition} 
We define \emph{the homogeneous Rademacher symbol} $\Psi_{p,q}^{\rm h}:\Gamma_{p,q}\to \Z$ by the homogenization of $\psi_{p,q}$, that is, we put 
\[ \Psi_{p,q}^{\rm h}(\gamma) = \lim_{n \to \infty} \frac{\psi_{p,q}(\gamma^n)}{n} = \lim_{n \to \infty} \frac{\Phi_{p,q}(\gamma^n)}{n} \] for every $\gamma \in \Gamma_{p,q}$. 
\end{definition} 
In comparison with Proposition \ref{prop.class-invariant}, for \emph{any} $\gamma, g \in \Gamma_{p,q}$ and $n\in \Z$, we have 
\[\Psi_{p,q}^{\rm h}(\gamma) = \Psi_{p,q}^{\rm h}(-\gamma) = - \Psi_{p,q}^{\rm h}(\gamma^{-1}) = \Psi_{p,q}^{\rm h}(g^{-1} \gamma g)\]
and $\Psi_{p,q}^{\rm h}(\gamma^n) = n \Psi_{p,q}^{\rm h}(\gamma)$. 

If $|\tr \gamma| \geq 2$, then $\Psi_{p,q}^{\rm h}(\gamma) = \Psi_{p,q} (\gamma)$ holds. 
If instead $|\tr \gamma| < 2$, then we have $\Psi_{p,q}^{\rm h}(\gamma) = 0$, while the original symbol satisfies 
\[
\Psi_{p,q}(S_p) = 
\begin{cases}
0 &\text{if } p = 2,\\
-q &\text{if } p > 2,
\end{cases} \ \ \ \ 
\Psi_{p,q}(U_q) = -p. 
\]
Note that we have $\tr S_p=2\cos \frac{\pi}{p}$ and $\tr U_q=2\cos \frac{\pi}{q}$. 

If $\tr \gamma \geq2$, then $\Psi_{p,q}^{\rm h}(\gamma)=\psi_{p,q}(\gamma)$ holds.

\begin{remark} Recently, in a view of the Manin--Drinfeld theorem, Burrin \cite{Burrin2020arXiv} introduced certain functions for a general Fuchsian group $\Gamma$ by using a recipe close to ours. Her functions may be seen as generalizations of our $\Phi_{p,q}$ and $\Psi_{p,q}$, for which our Theorem \ref{thm.cycleintegral} persist. She also proved that if $\Gamma$ is a non-cocompact Fuchsian group with genus zero, then the values of the functions are in $\Q$. 
Our result further claims for $\Gamma_{p,q}$ that the values are in $\Z$. 
\end{remark}


\section{Modular knots around the torus knot} \label{s4} 


In this section, we establish our main result, that is, the coincidence of the values of the Rademacher symbol and the linking number between modular knots and the $(p,q)$-torus knot. 

\subsection{The torus knot groups} \label{s4-1} 

Here, we prepare group theoretic lemmas, which enable us to clearly recognize the natural $\Z/r\Z$-cover 
$h:S^3-K_{p,q}\surj L(r,p-1)-\ol{K}_{p,q}$, as well as to make an explicit argument. 

Recall that the universal covering group $\wt{\SL_2}\R$ is the central extension of $\SL_2\R$ by $\Z$ corresponding to the 2-cocycle $W$, that is, $\wt{\SL_2}\R$ is $\SL_2\R\times \Z$ as a set and endowed with the multiplication 
\[ (\gamma_1, n_1) \cdot (\gamma_2, n_2) = (\gamma_1 \gamma_2, n_1 + n_2 + W(\gamma_1, \gamma_2)). \]
Let $P:\wt{\SL_2}\R\to \SL_2\R; (\gamma,n)\mapsto \gamma$ denote the natural projection and put $\wt{\Gamma}_{p,q}=P^{-1}(\Gamma_{p,q})$, so that we have $\wt{\Gamma}_{p,q}=\{(\gamma,n)\mid \gamma \in \Gamma_{p,q}\}< \wt{\SL_2}\R$. 
For each $\gamma \in \SL_2\R$, define the standard lift by $\wt{\gamma}=(\gamma,0)\in \wt{\SL_2}\R$. 
Then $\wt{\Gamma}_{p,q}$ is generated by $\wt{S}_p$ and $\wt{U}_q$, for which $\wt{S}_p\!^p=\wt{U}_q\!^q=(-I,1)$ holds.

Recall $r = pq - p - q$. We here explicitly define a discrete subgroup $G_r < \widetilde{\SL_2}\R$ by 
\[ G_r = \langle \widetilde{S}_p\!^r, \widetilde{U}_q\!^r \rangle = \langle \widetilde{S}_p\!^r, \widetilde{U}_q\!^r \mid (\widetilde{S}_p\!^r)^p = (\widetilde{U}_q\!^r)^q = (-I, 1)^r \rangle. \]
The following lemmas are due to Tsanov~\cite{Tsanov2013EM}. 
Since the original assertions are for $\PSL_2\R$, we partially attach proofs for later use.
For each group $G$, let $Z(G)$ denote the center, $[G,G]$ the commutator subgroup, and $G^{\rm ab}$ the abelianization. 

\begin{lemma}\label{lem.Gr-prop}
	\begin{itemize}
		\item[$(1)$] $Z(G_r)$ is generated by $(-I,1)^r= (-I, \frac{r+1}{2})$. 
		\item[$(2)$] $P(G_r) = \Gamma_{p,q}$.
	\end{itemize}
\end{lemma}

\begin{proof}
(1) An isomorphism $\widetilde{\Gamma}_{p,q}\congto G_r$ is defined by $\widetilde{S}_p \mapsto \widetilde{S}_p\!^r$ and $\widetilde{U}_q \mapsto \widetilde{U}_q\!^r$. 
Since $Z(\widetilde{\Gamma}_{p,q})$ is generated by $\widetilde{S}_p\!^p = \widetilde{U}_q\!^q = (-I,1)$, $Z(G_r)$ is generated by $(-I,1)^r= (-I, \frac{r+1}{2})$. 

(2) Since $r$ is an odd number coprime to both $p$ and $q$, there exist some $s,t \in \Z$ satisfying $rs \equiv 1 \pmod{2p}$ and $rt \equiv 1 \pmod{2q}$, hence $S_p\!^{rs} = S_p$ and $U_q\!^{rt} = U_q$. Thus we have $S_p, U_q \in P(G_r)$. 
\end{proof}

\begin{lemma} \label{abelianization}
	\begin{itemize}
		\item[{\rm (1)}] $[\widetilde{\Gamma}_{p,q}, \widetilde{\Gamma}_{p,q}]=[G_r, G_r]$.
		\item[{\rm (2)}] $\widetilde{\Gamma}_{p,q}\!^{\mathrm{ab}}\cong G_r\!^{\mathrm{ab}}\cong \Z$. 
		\item[{\rm (3)}] 
		$\widetilde{\Gamma}_{p,q}/G_r \cong \widetilde{\Gamma}_{p,q}\!^{\mathrm{ab}} / G_r\!^{\mathrm{ab}} \cong \Z/r\Z$.
	\end{itemize}
\end{lemma}

As mentioned in Section 1, we have the following. 
\begin{proposition}[\cite{RaymondVasquez1981}, \cite{Tsanov2013EM}]
{\rm (1)} 
 The spaces $\Gamma_{p,q}\backslash \SL_2\R\cong \wt{\Gamma}_{p,q} \backslash \wt{\SL_2}\R$ are homeomorphic to the exterior of a knot $\ol{K}_{p,q}$ in the lens space $L(r,p-1)$, where $\ol{K}_{p,q}$ is the image of a $(p,q)$-torus knot via the $\Z/r\Z$-cover $S^3\surj L(r,p-1)$. 

{\rm (2) 
} The space $G_r\backslash \wt{\SL_2}\R$ is homeomorphic to the exterior of the torus knot $K_{p,q}$ in $S^3$. 
\end{proposition} 
The second assertion was established by Raymond--Vasquez by using the theory of Seifert fibrations in \cite{RaymondVasquez1981}. 
Tsanov gave explicit homeomorphisms for both cases in \cite{Tsanov2013EM}. 
We remark that Tsanov discussed the lens space $L(r, p(q_1 - p_1 + pp_1))$ for a pair $(p_1, q_1) \in \Z^2$ with $pp_1 + qq_1 = 1$, which is homeomorphic to $L(r,p-1)$ by Brody's theorem. 

Since the fundamental group is given by $\pi_1(\Gamma_{p,q} \backslash \SL_2\R) \cong \pi_1(\widetilde{\Gamma}_{p,q} \backslash \widetilde{\SL_2}\R) \cong \widetilde{\Gamma}_{p,q}$, by the Hurewicz theorem and the lemmas above, we obtain the following. 

\begin{lemma} 
The groups $G_r\cong\pi_1(S^3-K_{p,q})$ are the kernels of any surjective homomorphism $\wt{\Gamma}_{p,q}\cong\pi_1(L(r,p-1)-\ol{K}_{p,q})\surj \Z/r\Z$. 
We may identify 
the corresponding $\Z/r\Z$-cover $h:S^3-K_{p,q}\to L(r,p-1)-\ol{K}_{p,q}$ with the natural surjection $G_r\backslash \wt{\SL_2}\R\surj \wt{\Gamma}_{p,q}\backslash \wt{\SL_2}\R$. 

The groups $G_r\!^{\rm ab}\cong H_1(S^3-K_{p,q};\Z)\cong \Z$ may be seen as the subgroups of $\wt{\Gamma}_{p,q}\!^{\rm ab}\cong H_1(L(r,p-1)-\ol{K}_{p,q};\Z)\cong \frac{\,1\,}{r}\Z$ of index $r$ in a natural way. 
\end{lemma}

The following diagram visualizes the situation. Here, for $G=\wt{\Gamma}_{p,q}$ and $G_r$, $G'$ denotes the commutator subgroup of $G$ and  
$Z'(G)$ denotes the subgroup of $Z(G)\cong \Z$ with index 2. The $\Z$-covers of $L(r,p-1)-\ol{K}_{p,q}$ and $S^3-K_{p,q}$ are denoted by $L_\infty=X_\infty$. 
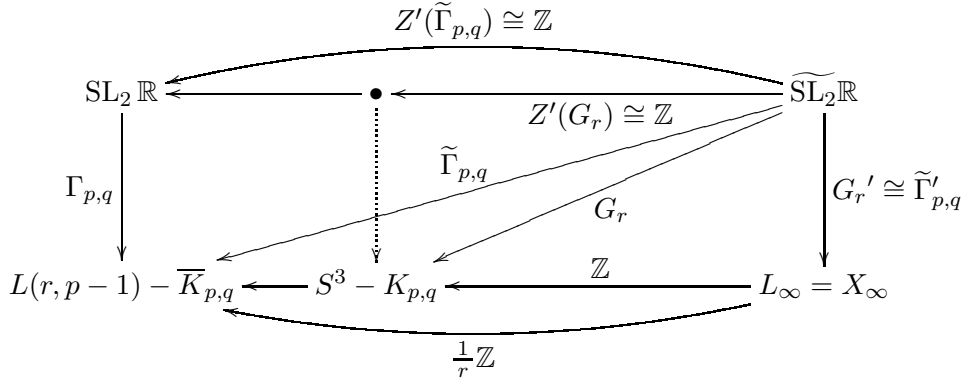
\[
\def\labelstyle{\normalsize}
\xymatrix{
	\SL_2\R \ar[dd]_{\Gamma_{p,q}} & \bullet \ar[l] \ar@{..>}[dd] & & & & \widetilde{\SL_2}\R \ar[llll]^{Z'(G_r) \cong \Z} \ar@/_18pt/[lllll]_{Z'(\widetilde{\Gamma}_{p,q}) \cong \Z} \ar[llllldd]_{\widetilde{\Gamma}_{p,q}} \ar[lllldd]^{G_r} \ar[dd]^{{G_r}' \cong {\widetilde{\Gamma}_{p,q}}'}\\
	 & & & & & \\
	L(r,p-1) - \overline{K}_{p,q} & S^3-K_{p,q} \ar[l] & & & & L_\infty = X_\infty \ar[llll]_{\Z} \ar@/^18pt/[lllll]^{\frac{\,1\,}{r}\Z}
}
\]

\subsection{Modular knots in the lens space} 
\label{s4-2} 

In this subsection, we introduce the notion of modular knots for $\Gamma_{p,q}$ around the $(p,q)$-torus knot in the lens space $L(r,p-1)$, recall the notions of the linking number and the winding number, 
and establish the former half of our main result on the linking number. 

\subsubsection{Modular knots} 

Let us first define a modular knot in the lens space. 
\begin{definition}\label{def.modularknot} 
{\rm (1)} Let $\gamma = \smat{a&b\\c&d} \in \Gamma_{p,q}$ be a primitive element with $a+d> 2$ and $c > 0$, so that $\gamma$ is diagonalized by the scaling matrix $M_\gamma$ and its larger eigenvalue satisfies $\xi_\gamma>1$. 
Define an oriented simple closed curve $C_\gamma(t)$ in $\Gamma_{p,q} \backslash \SL_2\R$ by
\[ C_\gamma(t) = M_\gamma \spmx{e^t & 0 \\ 0 & e^{-t}}, 
\quad (0 \leq t \leq \log \xi_\gamma). \] 
We call the image $C_\gamma$ in $\Gamma_{p,q} \backslash \SL_2\R \cong L(r, p-1) - \overline{K}_{p,q}$ with the induced orientation \emph{the modular knot associated to} $\gamma$. 
	
{\rm (2)} Let $\gamma \in \Gamma_{p,q}$ be any hyperbolic element, so that we have $\gamma=\pm \gamma_0^n$ for some primitive element $\gamma_0=\smat{a&b\\c&d} \in \Gamma_{p,q}$ with $a+d>2$ and $c>0$, and $n\in \Z$. 
We define \emph{the modular knot associated to} $\gamma$ by $C_\gamma=nC_{\gamma_0}$ with multiplicity. 
\end{definition}

\subsubsection{Linking numbers} 

A general theory of the linking number in a rational homology 3-sphere can be found in \cite[Section 77]{SeifertThrelfall1934}. 
Since $H_1(L(r,p-1);\Z)\cong \Z/r\Z$, 
the linking number in $L(r,p-1)$ takes value in $\frac{\,1\,}{r}\Z$. 
Via a standard homeomorphism $\Gamma_{p,q}\backslash \SL_2\R \congto T_1(\Gamma_{p,q}\backslash \bbH)$ to the unit tangent bundle, the knot $\ol{K}_{p,q}$ may be seen as the cusp orbit with a natural orientation. 
Let $\mu$ be a standard meridian of $\ol{K}_{p,q}$ and consider the isomorphism $H_1(L(r,p-1)-\ol{K}_{p,q};\Z)\congto \frac{\,1\,}{r}\Z$ sending $[\mu]$ to 1. 
A standard meridian $\mu$ may be explicitly given by the curve $c(t)$ in the proof of Proposition \ref{prop.tildeDelta} with $0\leq t\leq \lambda$. 
\begin{definition} 
The linking number $\mathrm{lk}(K, \overline{K}_{p,q})$ of an oriented knot $K$ in $L(r,p-1) - \overline{K}_{p,q}$ 
and the knot $\ol{K}_{p,q}$ is defined as 
the image of $[K]$ via the isomorphism $H_1(L(r,p-1)-\ol{K}_{p,q};\Z)\congto \frac{\,1\,}{r}\Z$. 

This definition naturally extends to a knot with multiplicity, that is, a formal sum of knots with coefficients in $\Z$. 
\end{definition}

\subsubsection{Winding numbers} 
In order to compute the linking number, let us recall the notion of the winding number. 
Let the unit circle $\mathbf{T} = \{|z| = 1\} \subset \C$ be endowed with the counter-clockwise orientation 
and let $H_1(\C^\times; \Z) \congto \Z$ denote the isomorphism sending $[\mathbf{T}]$ to 1. 

\begin{definition} \label{def.wind} 
For an oriented closed curve $C$ in $\C^\times$, the \emph{winding number} $\mathrm{ind}(C,0) \in \Z$ is defined to be the image of $[C]$ via the isomorphism $H_1(\C^\times; \Z) \congto \Z$. 
Equivalently, it is defined by the cycle integral as 
\[\mathrm{ind}(C,0) = \frac{\,1\,}{2\pi i} \int_C \frac{dz}{z}.\] 
\end{definition}
The equivalence of two definitions is verified by Cauchy's integral theorem. 

We define a lift $\widetilde{\Delta}_{p,q}: \SL_2 \R \to \C^\times$ of the cusp form $\Delta_{p,q}(z)$ by
\[ \widetilde{\Delta}_{p,q} (g) = j(g, i)^{-2pq} \Delta_{p,q}(g i). \]
Since $\Delta_{p,q}(z)$ has no zeros on $\bbH$ and satisfies $\widetilde{\Delta}_{p,q}(\gamma g) = \widetilde{\Delta}_{p,q}(g)$ for any $\gamma \in \Gamma_{p,q}$, 
we obtain the induced continuous function $\widetilde{\Delta}_{p,q}: \Gamma_{p,q} \backslash \SL_2\R \to \C^\times$. 

\begin{proposition}\label{prop.winding=psi} 
For a modular knot $C_\gamma$ defined in Definition \ref{def.modularknot} (1), we have
\[ \mathrm{ind}(\widetilde{\Delta}_{p,q}(C_\gamma), 0) = \psi_{p,q}(\gamma). \]
\end{proposition}

\begin{proof} 
Recall $\frac{d}{dz} \log \Delta_{p,q}(z) = 2\pi i r E_2^{(p,q)}(z)$ and put $z_0 = M_\gamma i$. 
Then by Theorem \ref{thm.cycleintegral}, we obtain 
\begin{equation*} \begin{aligned} 
\mathrm{ind}(\widetilde{\Delta}_{p,q}(C_\gamma), 0) &= \frac{\,1\,}{2\pi i} \int_{\widetilde{\Delta}_{p,q}(C_\gamma)} \frac{dz}{z} \\
&= \frac{\,1\,}{2\pi i} \int_0^{\log \xi_\gamma} \frac{d\widetilde{\Delta}_{p,q}(C_\gamma(t))}{\widetilde{\Delta}_{p,q}(C_\gamma(t))}\\
&= r \int_{z_0}^{\gamma z_0} E_2^{(p,q),*}(z) dz\\
&=\psi_{p,q}(\gamma).\\[-6mm] 
\end{aligned} \end{equation*} 
\end{proof} 

\begin{proposition} \label{prop.tildeDelta}
The function $\widetilde{\Delta}_{p,q}$ induces an isomorphism $H_1(\Gamma_{p,q} \backslash \SL_2\R;\Z) \congto H_1(\C^\times;\Z)$.
\end{proposition}

\begin{proof}
The function $\widetilde{\Delta}_{p,q}$ induces a group homomorphism $(\widetilde{\Delta}_{p,q})_* : H_1(\Gamma_{p,q} \backslash \SL_2\R;\Z) \to H_1(\C^\times;\Z)$. 
Since both homology groups are isomorphic to $\Z$, it suffices to show the surjectivity. 

If $(p,q) = (2,3)$, take a sufficiently large $y \in \R_{>0}$. Define a closed curve in $\SL_2\Z \backslash \SL_2\R$ by 
\[ C_y(t) = \pmat{1 & t \\ 0 & 1} \pmat{y^{1/2} & 0 \\ 0 & y^{-1/2}}, 
\quad (0 \leq t \leq 1) \]
and that in $\C^\times$ by 
\[ \widetilde{\Delta}_{2,3} (C_y(t)) = y^6 \Delta_{2,3} (t+iy), \quad (0 \leq t \leq 1). \]
Since $\Delta_{2,3}(z) = q_1 + O(q_1^2)$, we have $\mathrm{ind}(\widetilde{\Delta}_{2,3}(C_y(t)), 0) = 1$. Thus the map $(\widetilde{\Delta}_{2,3})_*$ is surjective.
	
If $(p,q) \neq (2,3)$, take the hyperbolic element $\gamma \in \Gamma_{p,q}$ defined in Lemma \ref{lem.psi=1}. By Proposition \ref{prop.winding=psi}, we have $\mathrm{ind}(\widetilde{\Delta}_{p,q}(C_\gamma),0) = \psi_{p,q}(\gamma) = 1$, which concludes that $(\widetilde{\Delta}_{p,q})_*$ is surjective.
\end{proof}

\subsubsection{Theorem in $L(r,p-1)$} 
By Proposition \ref{prop.tildeDelta},  
for any oriented knot $K$ in $L(r,p-1) - \overline{K}_{p,q} \cong \Gamma_{p,q} \backslash \SL_2\R$, we have 
\[ \mathrm{lk}(K, \overline{K}_{p,q}) = \frac{\,1\,}{r} \mathrm{ind}(\widetilde{\Delta}_{p,q}(K), 0). \]
Together with the results in Subsection \ref{subsec.variants}, we conclude the following.  
\begin{theorem} \label{thm.lens} 
{\rm (1)} 
Let $\gamma = \smat{a&b\\c&d} \in \Gamma_{p,q}$ be a primitive element with $a+d > 2$ and $c > 0$. Then the linking number of the modular knot $C_\gamma$ and the image $\overline{K}_{p,q}$ of the $(p,q)$-torus knot in the lens space $L(r,p-1)$ is given by
\[ \mathrm{lk}(C_\gamma, \overline{K}_{p,q}) = \frac{\,1\,}{r} \psi_{p,q}(\gamma). \]
	
{\rm (2)} Let $\gamma \in \Gamma_{p,q}$ be any hyperbolic element. 
Then the linking number is given by 
\[ \mathrm{lk}(C_\gamma, \overline{K}_{p,q}) = \frac{\,1\,}{r} \Psi_{p,q}(\gamma)= \frac{\,1\,}{r} \Psi_{p,q}^{\rm h}(\gamma). \]
\end{theorem}

\subsection{Modular knots in the 3-sphere} 
\label{s4-3} 

In this subsection, we investigate modular knots around the $(p,q)$-torus knot $K_{p,q}$ in $S^3$ to establish the latter half of our main theorem on the linking number. 

\subsubsection{Linking numbers in $\Z/r\Z$-cover} 

\begin{definition}  
For an oriented knot $K$ in $S^3 - K_{p,q}$, the \emph{linking number} $\mathrm{lk}(K, K_{p,q}) \in \Z$ is defined by the image of $[K]$ via the isomorphism $H_1(S^3 - K_{p,q}; \Z) \congto \Z$ sending a standard meridian $\mu$ of $K_{p,q}$ to 1. 
This definition naturally extends to knots with multiplicity. 
\end{definition} 

Recall that the restriction of the $\Z/r\Z$-cover $h:S^3\surj L(r,p-1)$ to the exterior of $K_{p,q}$ may be identified with the natural surjection $G_r \backslash \widetilde{\SL_2}\R \surj \wt{\Gamma}_{p,q} \backslash \wt{\SL_2}\R$. 
Let $K$ be an oriented knot in $L(r, p-1) - \overline{K}_{p,q}$ and $K'$ a connected component of $h^{-1}(K)$. 
The following two lemmas are consequences of a standard argument of the covering theory (e.g., the lifting property of continuous maps, \cite[Propositions 1.33, 1.34]{HatcherAT}). 
\begin{lemma} \label{lem.covering} 
The covering degree of the restriction $h : K'\to K$ coincides with the order of $[K]$ in $H_1(L(r,p-1); \Z)\cong \Z/r\Z$. 
The covering degree of $h: K_{p,q} \to \overline{K}_{p,q}$ is equal to $r$. 
\end{lemma} 

\begin{proof} Note that the decomposition group of $K'$ is a subgroup of the Deck transformation group ${\rm Deck}(h)\cong H_1(L(r,p-1);\Z)\cong \Z/r\Z$ generated by $[K]$. 
The assertion follows from the Hilbert ramification theory for $\Z/r\Z$-cover \cite[Section 2]{ueki1}. 
\end{proof} 

\begin{lemma}\label{lem.link-lens-link-S3} 
If $[K]$ in $H_1(L(r,p-1);\Z) \cong \Z/r\Z$ is of order $m$, then we have
\[
\mathrm{lk}(K', K_{p,q}) = m\, \mathrm{lk}(K, \overline{K}_{p,q}).
\]
\end{lemma}

\begin{proof} We have a connected surface $\Sigma$ in $L(r,p-1)$ with $\partial \Sigma=mK$ and 
a connected component $\Sigma'$ of the preimage $h^{-1}(\Sigma)$ with $\partial  \Sigma'=K'$. 
Let $\iota$ denote the intersection number. 
Then by Lemma \ref{lem.covering}, we have ${\rm lk}(K',K_{p,q})=\iota (\Sigma',K_{p,q})=\iota (\Sigma,\ol{K}_{p,q})= {\rm lk}(mK,\ol{K}_{p,q})=m\,{\rm lk}(K,\ol{K}_{p,q})$. 
\end{proof}

\subsubsection{Modular knots in $\Z/r\Z$-cover}

We define a modular knot in $S^3$ as a connected component of the inverse image of that in $L(r,p-1)$. 
\begin{definition} \label{def.modularknot'} 
(1) Let $\gamma=\smat{a&b\\c&d}\in \Gamma_{p,q}$ be a primitive element with $a+d>2$ and $c>0$. 
Consider the modular knot $C_\gamma$ in $L(r, p-1) - \overline{K}_{p,q}$ associated to $\gamma$ and 
let $m_\gamma$ denote the order of $[C_\gamma]$ in $H_1(L(r,p-1);\Z)\cong \Z/r\Z$,
so that the inverse image $h^{-1}(C_\gamma)$ consists of exactly $r/m_\gamma$-connected components. 
We call each connected component $C_\gamma'$ of $h^{-1}(C_\gamma)$ \emph{a modular knot associated to} $\gamma \in \Gamma_{p,q}$ in $S^3-K_{p,q}$. 

(2) Let $\gamma \in \Gamma_{p,q}$ be any hyperbolic element, so that we have $\gamma=\pm\gamma_0^\nu$ for some primitive $\gamma_0=\smat{a&b\\c&d} \in \Gamma_{p,q}$ with $a+d>2$ and $c>0$ and $\nu\in \Z$. 
Let $C'_{\gamma_0}$ be a modular knot in $S^3-K_{p,q}$ associated to $\gamma_0$. 
We call the knot $C_\gamma'=\nu C_{\gamma_0}'$ with multiplicity \emph{a modular knot associated to} $\gamma\in \Gamma_{p,q}$ in $S^3-K_{p,q}$. 
\end{definition} 

The following lemma plays a key role to explicitly find the integer $m_\gamma$. 

\begin{lemma} \label{lem.ngamma} \label{n-explicit-psi} 
For each $\gamma \in \Gamma_{p,q}$, we have $(\gamma,n)\in G_r$ if and only if 
\[ 2pq n \equiv \psi_{p,q}(\gamma) \ {\rm mod} \ r \]
holds. 
Such $n$'s define an element in $\Z/r\Z$. 

If $n_\gamma \in \Z$ with $(\gamma,n_\gamma)\in G_r$, then 
${\rm gcd}(r, n_\gamma)={\rm gcd}(r,\psi_{p,q}(\gamma))$ holds. 
\end{lemma} 

\begin{proof} 
By Lemma \ref{lem.Gr-prop} (2), there exitst some $n\in \Z$ satisfying $(\gamma, n) \in G_r$. 
In addition, by Lemma \ref{lem.Gr-prop} (1), we have $Z'(G_r)=P^{-1}(I) \cap G_r = \langle (I, r) \rangle = \langle (-I, 1)^{2r} \rangle $, which is the subgroup of $Z(G_r)\cong \Z$ with index 2. 
Now suppose $(\gamma,n), (\gamma,n') \in G_r$. Then we have $(\gamma, n) (\gamma, n')^{-1} \in G_r$, which implies $n-n'\equiv 0 \ {\rm mod}\ r$. 
Thus the set of $n\in \Z$ with $(\gamma, n) \in G_r$ defines a class $n_\gamma \in \Z/r\Z$. 
	
Now take $n_\gamma \in \Z$ with $(\gamma, n_\gamma)\in G_r$ for each $\gamma\in \Gamma_{p,q}$, so that we have a map $n_\bullet:\Gamma_{p,q}\to \Z$. 
Note that $\gcd(2pq, r) = 1$. Since $\Gamma_{p,q}$ is generated by $S_p$ and $U_q$ of orders $2p$ and $2q$, a group homomorphism $\Gamma_{p,q} \to \Z/r\Z$ is trivial, that is, we have $H^1(\Gamma_{p,q}; \Z/r\Z) = 0$. 
Since 
$(\gamma_1, n_{\gamma_1}) \cdot (\gamma_2, n_{\gamma_2}) = (\gamma_1 \gamma_2, n_{\gamma_1} + n_{\gamma_2} + W(\gamma_1, \gamma_2))$ in $G_r$, we have
 \[ n_{\gamma_1\gamma_2}\equiv n_{\gamma_1}+n_{\gamma_2}+W(\gamma_1,\gamma_2) \ {\rm mod}\ r. \] 
On the other hand, by Theorem \ref{thm.W}, we have 
\[ \psi_{p,q}(\gamma_1\gamma_2) = \psi_{p,q}(\gamma_1) + \psi_{p,q}(\gamma_2) + 2pq W(\gamma_1, \gamma_2). \]
Hence we have a group homomorphism $\psi_{p,q}(\gamma)-2pqn_\gamma \ {\rm mod}\ r: \Gamma_{p,q}\to \Z/r\Z$, which must be zero by $H^1(\Gamma_{p,q}; \Z/r\Z) = 0$. 
Thus we obtain $2pq n_\gamma \equiv \psi_{p,q}(\gamma) \ {\rm mod} \ r$. 
	
Again by ${\rm gcd}(2pq,r)=1$, we obtain ${\rm gcd}(r, n_\gamma)={\rm gcd}(r,\psi_{p,q}(\gamma))$. 
\end{proof}

Now let $\gamma = \smat{a&b\\c&d} \in \Gamma_{p,q}$ be a primitive element with $a+d > 2$ and $c > 0$ and 
take $n_\gamma\in \Z$ with $(\gamma,n_\gamma)\in G_r$. 
\begin{lemma} \label{lem.Cgammaell} 
For each $l \in \Z/r \Z$, we may define a simple closed curve in $G_r \backslash \widetilde{\SL_2}\R$ by 
\[ C_{\gamma,l}(t) = \big( M_\gamma \spmx{e^t & 0 \\ 0 & e^{-t}}, l \big), 
\quad (0 \leq t \leq \frac{r}{\gcd(r, \psi_{p,q}(\gamma))} \log \xi_\gamma). \]
\end{lemma}

\begin{proof}  
Note that we have 
\[\sgn(\gamma)>0, \ \ \sgn(M_\gamma \spmx{e^t&0\\0&e^{-t}})>0, \ \ \sgn(M_\gamma\spmx{e^{t+\log\xi_\gamma}&0\\0&e^{-t-\log \xi_\gamma}})>0.\]  
Then a direct calculation yields 
\begin{align*} 
C_{\gamma,l}(t+\log\xi_\gamma) 
&= \big( M_\gamma \spmx{\xi_\gamma&0\\0&\xi_\gamma^{-1}} \spmx{e^t & 0 \\ 0 & e^{-t}}, l \big) \\
& = (\gamma,0) \big( M_\gamma \spmx{e^t & 0 \\ 0 & e^{-t}}, l \big)\\
& = (\gamma,n_\gamma)(I,-n_\gamma)\big( M_\gamma \spmx{e^t & 0 \\ 0 & e^{-t}}, l \big)\\
& = (I,-n_\gamma)C_{\gamma,l}(t). 
\end{align*} 
Hence for any $k\in \Z$, we have 
\[C_{\gamma,l}(t+k\log \xi_\gamma)=(I,-kn_\gamma)C_{\gamma,l}(t).\] 
Since $Z'(G_r)=P^{-1}(I) \cap G_r = \langle (I, r) \rangle$, 
we have $(I,-kn_\gamma)\in G_r$ if and only if $-kn_\gamma =0$ in $\Z/r\Z$ holds. 
The least positive $k$ with $-kn_\gamma =0$ is given by 
$k=r/{\rm gcd}(r,n_\gamma)=r/{\rm gcd}(r,\psi_{p,q}(\gamma))$. 
Hence we obtain the assertion. 
\end{proof} 

The image $C_{\gamma,l}$ in $G_r\backslash \wt{\SL_2}\R\cong S^3-K_{p,q}$ with the induced orientation is a modular knot associated to $\gamma$. 
\begin{proposition} 
For $l,l' \in \Z/r\Z$, we have $C_{\gamma,l}=C_{\gamma,l'}$ if and only if $l\equiv l'\ {\rm mod}\ {\rm gcd}(r,\psi_{p,q}(\gamma))$ holds. 
The set of modular knots in $S^3-K_{p,q}$ associated to $\gamma$ coincides with 
$\{C_{\gamma,l}\mid l\in \Z/r\Z\}=\{C_{\gamma,l}\mid l=0,1,\cdots, {\rm gcd}(r,\psi_{p,q}(\gamma))-1\}$. 
\end{proposition}

\begin{proof} 
Suppose $C_{\gamma,l}=C_{\gamma,l'}$. Then there exists some $t\in \R_{>0}$ satisfying 
$C_{\gamma,l}(0)=C_{\gamma,l'}(t)$ in $G_r\backslash \wt{\SL_2}\R$, that is, 
there exists some $(\sigma,s)\in G_r$ satisfying 
\[(\sigma,s)(M_\gamma,l)=\big(M_\gamma \spmx{e^t & 0 \\ 0 & e^{-t}}, l' \big).\]
Since $\sigma M_\gamma=M_\gamma\smat{e^t & 0 \\ 0 & e^{-t}}$, 
there exists some $k\in \Z_{>0}$ satisfying 
$\sigma=\gamma^k$, $t=k\log \xi_\gamma$, and $s\equiv kn_\gamma$ mod $r$. Since 
\[(\gamma^kM_\gamma, kn_\gamma+l)=\big(M_\gamma\spmx{\xi_\gamma\!^k&0\\0&\xi_\gamma\!^{-k}},l'\big),\]
we have $kn_\gamma+l\equiv l'$ mod $r$. 
Hence we have $l\equiv l'$ mod ${\rm gcd}(r,n_\gamma)={\rm gcd}(r,\psi_{p,q}(\gamma))$. 

Suppose instead that $l\equiv l'$ mod ${\rm gcd}(r,n_\gamma)$. Then we have 
$l'=l+k\,{\rm gcd}(r,n_\gamma)$ and ${\rm gcd}(r,n_\gamma)=ar+b n_\gamma$ for some $k,a,b\in \Z$. 
By 
\begin{align*}
C_{\gamma,l'}(t)
&=\big(M_\gamma \spmx{e^t & 0 \\ 0 & e^{-t}}, l+akr+bkn_\gamma \big)\\
&=(I,bkn_\gamma)\big(M_\gamma \spmx{e^t & 0 \\ 0 & e^{-t}}, l \big)\\
&=C_{\gamma,l}(t-bk\log \xi_\gamma), 
\end{align*}
we obtain $C_{\gamma,l}=C_{\gamma,l'}$. 

Comparing the covering degree, we obtain the second assertion. 
\end{proof}

\begin{proposition} \label{prop.m_gamma} 
The element $[C_\gamma] \in H_1(L(r,p-1); \Z) \cong \Z/r\Z$ is of order 
\[m_\gamma = \dfrac{r}{\gcd(r, n_\gamma)} = \dfrac{r}{\gcd(r, \psi_{p,q}(\gamma))}.\]  
\end{proposition} 

\begin{proof} Since the period of $C_\gamma(t)$ is $\log \xi_\gamma$, Lemma \ref{lem.Cgammaell} yields that 
the covering degree of the restriction $h:C_{\gamma,l}\to C_\gamma$ is $r/\gcd(r, \psi_{p,q}(\gamma))$. 
By Lemma \ref{lem.covering}, we obtain the assertion. 
\end{proof} 

\subsubsection{Theorem in $S^3$}

By Lemma \ref{lem.link-lens-link-S3}, Theorem \ref{thm.lens}, and by Proposition \ref{prop.m_gamma}, we obtain \[ {\rm lk}(C_\gamma',K_{p,q})=m_\gamma\, {\rm lk}(C_\gamma,\ol{K}_{p,q})
=\dfrac{m_\gamma}{r}\psi_{p,q}(\gamma)
=\dfrac{1}{{\rm gcd}(r,\psi_{p,q}(\gamma))}\psi_{p,q}(\gamma).\] 

Together with the results in Subsection \ref{subsec.variants}, we conclude the following. 

\begin{theorem} \label{thm.S^3} 
{\rm (1)} 
Let $\gamma = \smat{a&b\\c&d} \in \Gamma_{p,q}$ be a primitive hyperbolic element with $\tr \gamma  > 2$ and $c > 0$. 
Then the linking number of each modular knot $C_\gamma'$ in $S^3-K_{p,q}$ associated to $\gamma$ and the $(p,q)$-torus knot $K_{p,q}$ is given by 
\[ \mathrm{lk}(C_\gamma', K_{p,q}) = \frac{\,1\,}{\gcd(r, \psi_{p,q}(\gamma))} \psi_{p,q}(\gamma). \]
	
{\rm (2)} Let $\gamma \in \Gamma_{p,q}$ be any hyperbolic element and $\gamma_0\in \Gamma_{p,q}$ a primitive element with $\gamma=\pm \gamma_0^\nu$ for some $\nu\in \Z$. Then the linking number is given by 
\[ \mathrm{lk}(C_\gamma', K_{p,q}) = \frac{\,1\,}{\gcd(r, \Psi_{p,q}(\gamma_0))} \Psi_{p,q}(\gamma)
=\frac{\,1\,}{\gcd(r, \Psi_{p,q}^{\rm h}(\gamma_0))} \Psi_{p,q}^{\rm h}(\gamma). \]
\end{theorem}

\begin{remark}\label{rem:Tsanov}
In above, we proved the theorem in $S^3$ via the case in the lens space. 
We may also directly discuss the case in $S^3$ by using automorphic differential forms of degree $1/r$ studied by Milnor~\cite[Section 5]{Milnor1975Brieskorn}. 
Indeed, we can construct a lift $\widetilde{\Delta}_{p,q}^{\,1/r} : G_r \backslash \widetilde{\SL_2}\R \to \C^\times$ satisfying $(\widetilde{\Delta}_{p,q}^{\,1/r}(\gamma, n))^r = \widetilde{\Delta}_{p,q}(\gamma)$ for every $(\gamma, n) \in \widetilde{\SL_2}\R$. By a similar argument, 
we may obtain
\[ \mathrm{lk}(C_\gamma', K_{p,q}) = \mathrm{ind}(\widetilde{\Delta}_{p,q}^{\,1/r} (C_\gamma'), 0) = \frac{\,1\,}{\gcd(r, \psi_{p,q}(\gamma) 
)} \psi_{p,q}(\gamma)\]
for $\gamma$ with the condition of Theorem \ref{thm.S^3} (1). 
The lift $\widetilde{\Delta}_{p,q}^{\, 1/r}$ equals Tsanov's function $\omega_\infty(z,dz)$ in \cite[Lemma 4.16]{Tsanov2013EM} up to a constant multiple, yielding a homeomorphism $G_r \backslash \widetilde{\SL_2}\R \cong S^3 - K_{p,q}$ \cite[Section 5]{Tsanov2013EM}. 
\end{remark}

\subsection{Euler cocycles} \label{subsec.Euler} 
In this subsection, we further introduce another variant $\Psi_{p,q}^{\rm e}$ of the Rademacher symbol as well as define knots corresponding to elliptic and parabolic elements, so that the theorems on linking numbers extends to whole $\Gamma_{p,q}$. 
This symbol is characterized by using an Euler cocycle, which arises as an obstruction to the existence of sections of cycles in the $S^1$-bundle $T_1 \Gamma_{p,q}\backslash\bbH \cong \Gamma_{p,q}\backslash \SL_2\R \cong L(r,p-1)-\ol{K}_{p,q}$. 
Our argument partially justifies Ghys's outlined second proof \cite[Section 3.4]{Ghys2007ICM} of his theorem.

\subsubsection{The linking numbers of fibers} \label{subsubsec.fibers} 
The singular fibers of the $S^1$-bundle 
corresponding to the elliptic points $a=e^{\pi i(1-1/p)}$ and $b=e^{\pi i/q}$ are parametrized as 
\[{\bf f}_a(t)=\spmx{1&-\cos \frac{\pi}{p}\\0&1} 
\spmx{(\sin \frac{\pi}{p})^{1/2}&0\\0& (\sin\frac{\pi}{p})^{-1/2}}
\spmx{\cos t&-\sin t\\ \sin t &\cos t}, \ (0\leq t\leq \dfrac{\pi}{p}),\]
\[{\bf f}_b(t)=\spmx{1&\cos \frac{\pi}{q}\\0&1}
\spmx{(\sin \frac{\pi}{q})^{1/2}&0\\0& (\sin\frac{\pi}{q})^{-1/2}}
\spmx{\cos t&-\sin t\\ \sin t &\cos t}, \ (0\leq t\leq \dfrac{\pi}{q}).\] 
Indeed, they define closed curves 
by 
\[{\bf f}_a(\frac{\pi}{p})=\spmx{0&-1\\1&2\cos \frac{\pi}{p}}{\bf f}_a(0)=S_p{\bf f}_a(0), \ {\bf f}_b(\frac{\pi}{q})=\spmx{2\cos \frac{\pi}{q}&-1\\1&0}{\bf f}_b(0)=U_q{\bf f}_b(0).\] 
In addition, for any $t\in \R$, we have ${\bf f}_a(t)i=a$, ${\bf f}_b(t)i=b$. 
By 
\[\wt{\Delta}_{p,q}({\bf f}_a(t))=j({\bf f}_a(t),i)^{-2pq}\Delta_{p,q}(a)=(\sin \frac{\pi}{p})^{pq}e^{-2pqit}\Delta_{p,q}(a),\] 
the winding number of $\wt{\Delta}_{p,q}({\bf f}_a(t))$ $(0\leq t\leq \frac{\pi}{p})$ around the origin is $-q$. 
In a similar way, the winding number of $\wt{\Delta}_{p,q}({\bf f}_b(t))$ $(0\leq t\leq \frac{\pi}{q})$ is $-p$. 
Thus by Lemma \ref{psi-generators}, we see that 
\[{\rm lk}({\bf f}_a,\ol{K}_{p,q})=\dfrac{-q}{r}=\frac{\psi_{p,q}(S_p)}{r}, \ \ {\rm lk}({\bf f}_b,\ol{K}_{p,q})=\dfrac{-p}{r}=\frac{\psi_{p,q}(U_q)}{r},\] 
and Theorem \ref{thm.lens} (1) for the Rademacher symbol $\psi_{p,q}$ may (literally) extends to these curves.

On the other hand, for any non-elliptic point $z=x+iy \in \bbH$,  
the corresponding fiber (a generic fiber) in $L(r,p-1)-\ol{K}_{p,q}\cong \Gamma_{p,q}\backslash \SL_2\R$ is parametrized as
\[{\bf f}_z(t)=\spmx{1&x\\0&1} \spmx{y^{1/2}&0\\0&y^{-1/2}} \spmx{\cos t&-\sin t\\ \sin t&\cos t}, \ (0\leq t\leq \pi).\] 
Indeed, we have ${\bf f}_z(\pi)=-{\bf f}_z(0)={\bf f}_z(0)$ and ${\bf f}_z(t)i=z$. By 
\[\wt{\Delta}_{p,q}({\bf f}_z(t))=j({\bf f}_z(t),i)^{-2pq}\Delta_{p,q}(z)=e^{-2pqit}y^{pq}\Delta_{p,q}(z),\] 
the winding number of $\wt{\Delta}_{p,q}({\bf f}_z(t))$ $(0\leq t\leq \pi)$ around the origin is ${\rm ind}(\wt{\Delta}_{p,q}({\bf f}_z),0)=-pq$. Hence the linking number of a generic fiber is given by 
\[{\rm lk}({\bf f}_z,\ol{K}_{p,q})=\frac{-pq}{r}.\] 

\subsubsection{Knots for $S_p$, $U_q$, and $T_{p,q}$} 
In order to extend the the theorems on linking numbers to whole $\Gamma_{p,q}$, we define knots corresponding to elliptic and parabolic elements. 
Take a sufficiently small $\varepsilon \in \R_{>0}$. For the elliptic point $a=e^{\pi i(1-1/p)}$, we consider a circle 
\[\tilde{c}_a=\{z\in \bbH\mid d_{\rm hyp} (a,z)=\varepsilon\}\]
with a clockwise orientation, where $d_{\rm hyp}$ denotes the hyperbolic distance on $\bbH$. 
The elliptic element $S_p$ acts on $\tilde{c}_a$ as a rotation of angle $-2\pi/p$. 
Take any point $z_0\in \tilde{c}_a$ and let $\ol{s}_a$ denotes the circle segment connecting $z_0$ to $S_pz_0$. 
Then the image $c_a$ of $\ol{s}_a$ in $\Gamma_{p,q}\backslash \bbH$ is a simple closed curve. 
In addition, take any point $Z_0 \in \SL_2\R$ with $Z_0i=z_0$ and let $s_a$ denote the section of $\ol{s}_a$ connecting $Z_0$ to $S_pZ_0$. 
Then the image $C_a$ of $s_a$ in $\Gamma_{p,q}\backslash \SL_2\R\cong L(r,p-1)-\ol{K}_{p,q}$ is a simple closed curve 
satisfying $C_a i= c_a$. 
Since $C_a \to {\bf f}_a$ as $\varepsilon \to 0$, we have 
\[{\rm lk}(C_a, \ol{K}_{p,q})={\rm lk}({\bf f}_a,\ol{K}_{p,q})=\frac{-q}{r}.\] 
Similarly, for $b=e^{\pi i/q}$, we define simple closed curves $c_b$ and $C_b$ 
satisfying $C_bi=c_b$ and 
\[{\rm lk}(C_b,\ol{K}_{p,q})={\rm lk}({\bf f}_b, \ol{K}_{p,q})=\frac{-p}{r}.\]

For the parabolic element $T_{p,q}$, as in the proof of Proposition \ref{prop.tildeDelta} for $(p,q)=(2,3)$, we take a lift $C_y(t)$ $(0\leq t\leq \lambda=2(\cos \frac{\pi}{p} + \cos \frac{\pi}{q}))$ of a holocycle so that we have 
\[{\rm lk}(C_y, \ol{K}_{p,q})=1=\frac{\,r\,}{r}=\frac{\,1\,}{r}\psi_{p,q}(T_{p,q}).\] 

\subsubsection{Theorem on whole $\Gamma_{p,q}$} 
Note that the fundamental group of the orbifold $\Gamma_{p,q}\backslash \bbH$ is described by both the languages of loops and covering spaces (cf.\cite[Chapter 13]{Ratcliffe2019book}). 
For each $\gamma \in \Gamma_{p,q}$, let $w$ be a fixed point on $\bbH\cup \R \cup \{i\infty\}$ and consider the stabilizer $(\Gamma_{p,q})_w$. 
If $\gamma$ is hyperbolic or parabolic, then $(\Gamma_{p,q})_w \cong \Z\times \Z/2\Z$. 
If instead $\gamma$ is elliptic, then $(\Gamma_{p,q})_w$ is a finite cyclic group. 
Let $\tilde{c}$ be a curve in $\bbH$ which is stable under the action of $(\Gamma_{p,q})_w$ and let $c$ denote the image of $\tilde{c}$ in 
$\Gamma_{p,q}\backslash \bbH$. 
If $\gamma$ is elliptic, then $c$ is a cycle around a cone point. 
If $\gamma$ is parabolic, then $c$ is the image of a holocycle. 
If $\gamma$ is hyperbolic, then we further assume that $\tilde{c}$ is a geodesic. 
Such $c$ is freely homotopic to a generator of $\gamma$ in the sense of the orbifold fundamental group. 

We define the knot $C_\gamma$ as a section of such $c$. 
More precisely, in addition to Definition \ref{def.modularknot}, 
we define knots corresponding to elliptic and parabolic elements as follows. 

\begin{definition} 
We put $C_{S_p}=C_a$, $C_{U_q}=C_b$, and $C_{T_{p,q}}=C_y$ discussed in above. 
In addition, for any $g\in \Gamma_{p,q}$, we put $C_{\pm g^{-1}S_p^{\,n}g}=nC_{S_p}$ for $n=1,2,\cdots, p-1$ and 
$C_{\pm g^{-1}U_q^{\,n}g}=nC_{U_q}$ for $n=1,2,\cdots ,q-1$. 
For any $g\in \Gamma_{p,q}$ and $n\in \Z$, we put $C_{\pm g^{-1}T_{p,q}^{\,n}g}=nC_{T_{p,q}}$. 
\end{definition} 

\begin{definition} 
We define \emph{the modified Rademacher symbol} $\Psi_{p,q}^{\rm e}:\Gamma_{p,q}\to \Z$ by 
\begin{align*}
	\Psi_{p,q}^{\rm e} (\gamma) = \begin{cases}
			-nq &\text{if } \gamma \sim \pm S_p^{\,n}\ (1 \leq n \leq p-1),\\
			-np &\text{if } \gamma \sim \pm U_q^{\,n}\ (1 \leq n \leq q-1),\\
			\Psi_{p,q}(\gamma) = \Psi_{p,q}^\mathrm{h}(\gamma) &\text{if otherwise}.
		\end{cases}
\end{align*}
where $\sim$ denotes the group conjugate in $\Gamma_{p,q}$. 
\end{definition} 

We remark that $\Psi_{p,q}^{\rm e} (\gamma)=\psi_{p,q}(\gamma)$ holds 
if $\tr \gamma \geq 2$ or $\gamma =S_p^{\,n}$ ($1\leq n\leq p-1$) or $\gamma=U_q^{\, n}$ ($1\leq n\leq q-1$). 
By combining all above, 
we may conclude the following. 

\begin{theorem} \label{thm.wholeGamma} 
For any $\gamma \in \Gamma_{p,q}$, the linking number in $L(r,p-1)$ is given by 
\[{\rm lk}(C_\gamma,\ol{K}_{p,q})=\frac{\,1\,}{r}\Psi_{p,q}^{\rm e}(\gamma).\] 
In addition, 
suppose that $\gamma=\pm \gamma_0^{\,\nu}$ for a primitive non-elliptic element $\gamma_0\in \Gamma_{p,q}$ and $\nu \in \Z$ 
or $\gamma \sim \pm S_p^{\, n}$ $(1\leq n\leq p-1)$ or  $\gamma \sim \pm U_q^{\, n}$ $(1\leq n\leq q-1)$. 
If $C'_\gamma$ is a connected component of $h^{-1}(C_\gamma)$ in the sense of Definition \ref{def.modularknot'} (2), then the linking number in $S^3$ is given by 
\[{\rm lk}(C'_\gamma, K_{p,q})=\frac{\,1\,}{{\rm gcd}(r,\Psi_{p,q}^{\rm e}(\gamma_0))}\Psi_{p,q}^{\rm e}(\gamma).\] 
\end{theorem}

\subsubsection{An Euler cocycle for $\Psi_{p,q}^{\rm e}$} 
Let ${\bf f}={\bf f}_z$ be a generic fiber given in Section \ref{subsubsec.fibers}. 
An Euler cocycle ${\rm eu}:\Gamma_{p,q}^{\,2}\to \Z$ of the $S^1$-bundle $T_1\Gamma_{p,q}\backslash \bbH \cong L(r,p-1)-\ol{K}_{p,q}$ is defined by the equality 
\[[C_{\gamma_1\gamma_2}]-[C_{\gamma_1}]-[C_{\gamma_2}]=-{\rm eu}(\gamma_1,\gamma_2)[{\bf f}]\] 
in $H_1(L(r,p-1)-\ol{K}_{p,q};\Z)$ for every $\gamma_1,\gamma_2\in \Gamma_{p,q}$. 
Taking the linking numbers with $\ol{K}_{p,q}$, we obtain 
\[{\rm lk}(C_{\gamma_1\gamma_2}, \ol{K}_{p,q}) - {\rm lk}(C_{\gamma_1}, \ol{K}_{p,q}) -{\rm lk}(C_{\gamma_2}, \ol{K}_{p,q})= -{\rm eu}(\gamma_1,\gamma_2) {\rm lk}({\bf f}, \ol{K}_{p,q}) = {\rm eu}(\gamma_1,\gamma_2) \frac{pq}{r}.\]

Note that we have $H^2(\Gamma_{p,q}/\{\pm I\};\Z)\cong \Z/pq\Z$ and $C_{\gamma}=C_{-\gamma}$ for any $\gamma \in \Gamma_{p,q}$. Let $\phi:\Gamma_{p,q} \to \Z$ be a unique function satisfying $-\delta \phi=pq {\rm eu}$ and $\phi(\gamma)=\phi(-\gamma)$ for any $\gamma \in \Gamma_{p,q}$. Then for any $\gamma \in \Gamma_{p,q}$, we have ${\rm lk}(C_\gamma, \ol{K}_{p,q})=\phi(\gamma)/r$. 
Together with the equality ${\rm lk}(C_\gamma, \ol{K}_{p,q})=\Psi_{p,q}^{\rm e}(\gamma)/r$ in Theorem \ref{thm.wholeGamma}, 
we obtain the following. 
\begin{theorem} \label{thm.Euler} 
Let ${\rm eu}:\Gamma_{p,q}^{\,2}\to \Z$ denote the Euler cocycle function defined as above. Then the modified Rademacher symbol $\Psi_{p,q}^{\rm e}$ is a unique function satisfying $-\delta \Psi_{p,q}^{\rm e}= pq {\rm eu}$ and $\Psi_{p,q}^{\rm e}(\gamma)=\Psi_{p,q}^{\rm e}(-\gamma)$ for any $\gamma\in \Gamma_{p,q}$. 
\end{theorem}

\begin{remark}
We may replace $\Psi_{p,q}^{\rm e}$ and ${\rm eu}$ in Theorem \ref{thm.Euler} by $\psi_{p,q}$ and $W$ 
by modifying the definition of modular knots for $\gamma$'s which do not satisfy the condition of Theorem \ref{thm.lens} (1). 
In this case, the equalities $C_\gamma=C_{-\gamma}$ and $C_{\gamma^n}=nC_\gamma$ will be modified according to the formula $\psi_{p,q}(-\gamma) = \psi_{p,q}(\gamma) + pq \sgn(\gamma)$. 
\end{remark}

\begin{remark} 
Ghys claims in \cite[Section 3.4]{Ghys2007ICM} that if we adapt the definition of modular knots to parabolic and elliptic elements, then his theorem follows from results of Atiyah \cite{Atiyah1987MA} and Barge--Ghys \cite{BargeGhys1992}, 
which explicitly investigate Euler cocycles in a view of ${\rm Homeo}^+S^1$. 
If we directly extend the results of Atiyah and Barge--Ghys for $\Gamma_{p,q}$, 
then we may obtain alternative proofs of our theorems on the linking numbers. 
\end{remark} 

\section{Miscellaneous} \label{s5}


Finally, we give some remarks and 
further problems. 

\subsection{Templates and codings} \label{ss.templates}
Ghys gave three proofs for his theorem on the Rademacher symbol for $\SL_2\Z$ and the linking number around the trefoil. 
In this article, through Sections 2--4, we generalized his first proof in \cite[Section 3.3]{Ghys2007ICM} by introducing the cusp form $\Delta_{p,q}(z)$, as well as discussed an Euler cocycle in a view of his second outlined proof in  \cite[Section 3.4]{Ghys2007ICM}.

Ghys's third proof in \cite[Section 3.5]{Ghys2007ICM} is a dynamical approach. A Lorenz knot is a periodic orbit appearing in the Lorenz attractor. Ghys proved for $\SL_2\Z$ that isotopy classes of Lorenz knots and modular knots coincide. 
In addition, he gave an explicit formula for $\mathrm{lk}(C_\gamma, K_{2,3})$ by using the Lorenz template. 
A hyperbolic element $\gamma \in \SL_2 \Z$ is conjugate to a matrix of the form
\[ \gamma \sim \pm S_2 U_3^{\varepsilon_1} S_2 U_3^{\varepsilon_2} \cdots S_2 U_3^{\varepsilon_n} \]
with $\varepsilon_i \in \{+1, -1\}$. Then, the linking number counts the number of left and right codes on the Lorenz template, that is, 
\[ \mathrm{lk}(C_\gamma, K_{2,3}) = \sum_{i=1}^n \varepsilon_i. \] 
On the other hand, Rademacher showed in \cite[(70)]{RademacherGrosswald1972} that $\sum_{i=1}^n \varepsilon_i=\Psi_{2,3}(\gamma)$.
Thus we obtain $\mathrm{lk}(C_\gamma, K_{2,3}) = \Psi_{2,3}(\gamma)$. 

The templates for geodesic flows for triangle groups are studied by Dehornoy and Pinsky 
\cite{Pinsky2014ETDS, Dehornoy2015AGT, DehornoyPinsky2018ETDS}. 
In particular, Dehornoy \cite[Proposition 5.7]{Dehornoy2015AGT} gave an explicit formula for the linking number between a periodic orbit of the geodesic flow $\Phi_{\Gamma_{p,q} \backslash \bbH}$ and the $(p,q)$-torus knot $\overline{K}_{p,q}$. 
By combining their result and Theorem \ref{thm.lens}, we may obtain an explicit formula of the Rademacher symbol $\Psi_{p,q}(\gamma)$. 
On the other hand, if one can show the explicit formula of $\Psi_{p,q}(\gamma)$ directly from the definition, then we obtain a generalization of Ghys's third proof.

\subsection{Distributions}
It is a natural question to ask the relation between the linking number ${\rm lk}(C_\gamma,K_{p,q})$ of a modular knot and the length $\ell(C_\gamma)$ of the corresponding closed geodesic on the modular orbifold. 
Based on Sarnak's idea in his letter \cite{Sarnak2010CMA}, Mozzochi \cite{Mozzochi2013Israel} proved variants of prime geodesic theorems to establish the following distribution formula, invoking the Selberg trace formula for 
$\SL_2\Z$; 
\begin{proposition} \label{prop.SM}
Suppose that $\gamma$ runs through conjugacy classes of primitive hyperbolic elements in $\SL_2\Z$ with $\tr \gamma > 2$ and let $\ell(\gamma)=2\log \xi_\gamma$ denote the length of the image of each modular knot $C_\gamma$ in $\SL_2 \Z \backslash \bbH$. Then 
for each $-\infty \leq a \leq b \leq \infty$, we have 
\[
	\lim_{y\to \infty}\frac{\#\{ \gamma \mid \ell(C_\gamma) \leq y,\ a \leq \dfrac{\mathrm{lk}(C_\gamma, K_{2,3})}{\ell(C_\gamma)} \leq b\}}{\#\{\gamma \mid \ell(C_\gamma) \leq y\}} = \frac{\arctan \dfrac{\pi b}{3} - \arctan \dfrac{\pi a}{3}}{\pi}.
\]
\end{proposition}

Von Essen generalized their results in his PhD thesis~\cite{vonEssenPhD} for any cofinite Fuchsian group with a multiplier system;
Let $\Gamma < \SL_2\R$ be a cofinite Fuchsian group, let $f: \bbH \to \C$ be a holomorphic modular form of weight $1$ for $\Gamma$ with no zero on $\bbH$, and let $\nu: \Gamma \to \C$ be a multiplier system, namely, we have 
\[
	f(\gamma z) = \nu(\gamma) j(\gamma, z) f(z)
\]
for every $\gamma \in \Gamma$. For its harmonic logarithm $F(z) = \log f(z)$, define $\Phi: \Gamma \to \C$ by
\[
	F(\gamma z) - F(z) = \log j(\gamma, z) + 2\pi i \Phi(\gamma).
\]
Assume in addition that the image of $\Phi$ is contained in $\Q$. 
By invoking the Selberg trace formula for Fuchsian groups, von Essen gave generalizations of Sarnak--Mozzochi's results. 
For instance, his Theorem H implies the following. 
\begin{proposition} \label{prop.vE}
If we replace $\SL_2\Z$ by $\Gamma$ in Proposition \ref{prop.SM}, then we have 
\[
	\lim_{y \to \infty} \frac{\#\{\gamma \mid \ell(\gamma) \leq y, a \leq \dfrac{\Phi(\gamma)}{\ell(\gamma)} \leq b\}}{\#\{\gamma \mid \ell(\gamma) \leq y\}} = \frac{\arctan 4\pi b - \arctan 4\pi a}{\pi}.
\]
\end{proposition} 
We remark that von Essen also showed for the Hecke triangle group $H_n = \Gamma_{2,n}$ a formula which is essentially the same as in our Theorem \ref{thm.lens} (1). 
His construction of the cusp form $\Delta_{2,n}(z)$ differs from ours but 
is closely related to Tsanov's construction of $\omega_\infty(z,dz)$ explained in Remark \ref{rem:Tsanov}. 

His results and Proposition \ref{prop.vE} are applicable to our setting with a more general triangle group $\Gamma_{p,q}$. In fact, 
%
let $f(z) = \Delta_{p,q}(z)^{1/2pq} = \exp \frac{1}{2pq} F_{p,q}(z)$ and $F(z) = \frac{1}{2pq} F_{p,q}(z)$. By Definition \ref{def.psi}, we have
\[
	f(\gamma z) = \nu(\gamma) j(\gamma, z) f(z), \quad \nu(\gamma) = e^{2\pi i \frac{\psi_{p,q}(\gamma)}{2pq}},
\]
and $\Phi(\gamma) = \frac{1}{2pq} \psi_{p,q}(\gamma)$. 
Thus, we obtain the following. 
\begin{corollary} If we replace $\SL_2\Z$ by $\Gamma_{p,q}$ in Proposition \ref{prop.SM}, then we have 
\[
	\lim_{y \to \infty} \frac{\#\{\gamma \mid \ell(\gamma) \leq y, a \leq \dfrac{\psi_{p,q}(\gamma)}{\ell(\gamma)} \leq b\}}{\#\{\gamma \mid \ell(\gamma) \leq y\}} = \frac{\arctan \dfrac{2\pi b}{pq} - \arctan \dfrac{2\pi a}{pq}}{\pi}.
\]
\end{corollary} 
By our Theorem \ref{thm.lens}, we may replace $\psi_{p,q}(\gamma)$ by $r {\rm lk}(C_\gamma,\ol{K}_{p,q})$ to obtain the Sarnak--Mozzochi formula for $\Gamma_{p,q}$. 


\begin{remark} 
The set of modular knots around the trefoil satisfies another distribution formula called \emph{the Chebotarev law} in the sense of Mazur \cite{Mazur2012} and McMullen \cite{McMullen2013CM}, so that it may be seen as an analogue of the set of all prime numbers in ${\rm Spec}\Z$ \cite{Ueki7, Ueki9}, in a sense of arithmetic topology \cite{Morishita2012}. 
An exploration of a unified viewpoint for these formulas would be of further interest.
\end{remark}



\subsection{Further problems} 

\subsubsection{Hyperbolic analogue} 
Duke--Imamo\={g}lu--T\'{o}th \cite{DukeImamogluToth2017Duke} investigated the linking number of two modular knots for $\SL_2\Z$. 
More precisely, they introduced a hyperbolic analogue of the Rademacher symbol $\Psi_\gamma(\sigma)$ for two hyperbolic elements $\gamma, \sigma \in \SL_2 \Z$ by using rational period functions, and established the equation $\Psi_\gamma(\sigma) = \mathrm{lk}(C_\gamma^+ + C_\gamma^-, C_\sigma^+ + C_\sigma^-)$. 
Here $C_\gamma^+$ is the modular knot as before, and $C_\gamma^-$ is another knot such that $C_\gamma^+ + C_\gamma^-$ is null-homologous in $S^3 - K_{2,3}$. 
Furthermore, the first author \cite{Matsusaka2020-arXiv} gave an explicit formula for the hyperbolic Rademacher symbol $\Psi_\gamma(\sigma)$ in terms of the coefficients of the continued fraction expansion of the fixed points of $\gamma$ and $\sigma$. 
An open question for $\SL_2\Z$ is to find a modular object yielding the linking number $\mathrm{lk}(C_\gamma, C_\sigma)$ (see also \cite{Rickards2021NT}). 
We may expect similar results for general triangle groups $\Gamma(p,q,r)$. 

\subsubsection{Other characterizations} 
In \cite[Theorem 5.60]{Atiyah1987MA}, Atiyah gave seven different definitions of the Rademacher symbol for hyperbolic elements of $\SL_2\Z$ (see also \cite{BargeGhys1992}). 
It would be interesting to extend any of them for $\Gamma_{p,q}$. 

\subsubsection{Galois actions} 
Since torus knots are algebraic knots, we have a natural action of the absolute Galois group on the profinite completions of the knot groups. 
We wonder if we may, in a sense, parametrize the Galois action via modular knots. 

\subsection*{Acknowledgments} 
The authors would like to express their sincere gratitude to Masanobu Kaneko for
his introduction to Asai's work in a private seminar and to Masanori Morishita for posing an interesting question related to Ghys's work. 
The authors are also grateful to Pierre Dehornoy, Kazuhiro Ichihara, \"{O}zlem Imamo\={g}lu, Morimichi Kawasaki, Ulf K\"{u}hn, Shuhei Maruyama, Makoto Sakuma, Yuji Terashima, and Masahito Yamazaki for useful information and fruitful conversations. 
Furthermore, the authors would like to thank all the participants who joined the online seminar FTTZS throughout the COVID-19 situation for cheerful communication. 
The first and the second authors have been partially supported by JSPS KAKENHI Grant Number JP20K14292 and JP19K14538 respectively. 

\bibliographystyle{amsalpha}
\bibliography{MatsusakaUeki-arXiv.bbl} 

\end{document}